\documentclass[11pt]{article}

\usepackage{fullpage,times,url,natbib}

\usepackage{amsthm,amsfonts,amsmath,amssymb,epsfig,color,float,graphicx,verbatim}
\usepackage{algorithm,algorithmic}

\usepackage{hyperref}
\hypersetup{
	colorlinks   = true, 
	urlcolor     = blue, 
	linkcolor    = blue, 
	citecolor   = black 
}

\newtheorem{theorem}{Theorem}
\newtheorem{proposition}{Proposition}
\newtheorem{lemma}{Lemma}

\newcommand{\reals}{\mathbb{R}}

\newcommand{\be}{\mathbf{e}}
\newcommand{\bx}{\mathbf{x}}
\newcommand{\bw}{\mathbf{w}}

\newcommand{\bu}{\mathbf{u}}
\newcommand{\bv}{\mathbf{v}}

\newcommand{\br}{\mathbf{r}}

\newcommand{\Ocal}{\mathcal{O}}

\newcommand{\norm}[1]{\|#1\|}
\newcommand{\inner}[1]{\langle#1\rangle}

\newcommand{\subsecref}[1]{Subsection~\ref{#1}}

\renewcommand{\eqref}[1]{Eq.~(\ref{#1})}
\newcommand{\lemref}[1]{Lemma~\ref{#1}}

\newcommand{\thmref}[1]{Thm.~\ref{#1}}
\newcommand{\propref}[1]{Proposition~\ref{#1}}
\newcommand{\appref}[1]{Appendix~\ref{#1}}

\title{Oracle Complexity of Second-Order Methods\\ for Smooth 
Convex 
Optimization}
\author{Yossi Arjevani\qquad Ohad Shamir \qquad Ron Shiff\\
Department of 
Computer Science and 
Applied 
Mathematics\\
Weizmann Institute of Science\\
\texttt{\{yossi.arjevani,ohad.shamir\}@weizmann.ac.il}\\
\texttt{ron.shiff1@gmail.com}}
\date{}

\begin{document}

\maketitle

\begin{abstract}
Second-order methods, which utilize gradients as well as Hessians to optimize a 
given function, are of major importance in mathematical optimization. In this 
work, we prove tight bounds on the oracle complexity of such methods for smooth 
convex functions, or equivalently, the worst-case number of iterations required 
to optimize such functions to a given accuracy. In particular, these bounds
indicate when such methods can or cannot improve on gradient-based methods, 
whose oracle complexity is much better understood. We also provide 
generalizations of our results to higher-order methods.
\end{abstract}

\section{Introduction}

We consider an unconstrained optimization problem of the form
\begin{equation}\label{eq:obj}
\min_{\bw\in \reals^d} f(\bw),
\end{equation}
where $f$ is a generic smooth and
convex function. A natural and fundamental question is how efficiently can we 
optimize such functions. 

We study this question through the well-known framework 
of oracle complexity \citep{YudNem83}, which focuses on iterative methods 
relying on local information. Specifically, it is assumed that the algorithm's 
access to the function $f$ is limited to an oracle, which 
given a point $\bw$, returns the values and derivatives of the function $f$ at 
$\bw$. This naturally models standard optimization approaches to 
unstructured problems such as \eqref{eq:obj}, and allows one to study their 
efficiency, by bounding the number of oracle calls required to reach a given 
optimization error. Different classes of methods can be distinguished by the 
type of oracle they use. For example, gradient-based methods (such as gradient 
descent or accelerated gradient descent) rely on a first-order oracle, which 
returns gradients, whereas methods such as the Newton method rely on a 
second-order oracle, which returns gradients as well as Hessians. 

The theory of \emph{first-order} oracle complexity is quite well developed 
\citep{YudNem83,nesterov2004introductory,nemirovski2005efficient}. For example, 
if the dimension is unrestricted, $f$ in \eqref{eq:obj} has 
$\mu_1$-Lipschitz gradients, and the algorithm makes its first oracle query at 
a point $\bw_1$, then the worst-case number of queries $T$ required to attain a 
point $\bw_T$ satisfying $f(\bw_T)-\min_{\bw}f(\bw)\leq \epsilon$ is
\begin{equation}\label{eq:firstorder}
\Theta\left(\sqrt{\frac{\mu_1 D^2}{\epsilon}}\right)~,
\end{equation}
where $D$ is an upper bound on the distance between $\bw_1$ and the nearest 
minimizer of $f$. Moreover, if the function $f$ is also 
$\lambda$-strongly 
convex for some $\lambda>0$\footnote{Assuming $f$ is twice-differentiable, this 
corresponds to 
$\nabla^2 f(\bw)\succeq \lambda I$ uniformly for all $\bw$.}, then the oracle 
complexity 
bound is
\begin{equation}\label{eq:firstorder_strong}
\Theta\left(\sqrt{\frac{\mu_1}{\lambda}}\cdot 
\log\left(\frac{\mu_1 D^2}{\epsilon}\right)\right)~.
\end{equation}
Both bounds are achievable using accelerated gradient descent 
\citep{nesterov1983method}.

However, these bounds do not capture 
the attainable performance of \emph{second-order} methods, which rely on 
gradient as well as Hessian information. This is a central class of 
optimization methods, including the well-known Newton method and its many 
variants. Clearly, since these methods rely on Hessians as well as gradients, 
their oracle complexity can only be better than first-order methods. 
On the flip side, the per-iteration computational complexity is generally 
higher,
in order to process the additional 
Hessian information (especially in high-dimensional problems where the Hessian 
matrix may be very large). Thus, it is natural to ask how much does this 
added per-iteration complexity pay off in terms of oracle complexity. 

To answer this question, one needs good oracle complexity 
lower bounds for second-order methods, which establish the limits of 
attainable performance using any such algorithm. Perhaps surprisingly, such 
results do not seem to currently exist in the literature, and clarifying the 
oracle complexity of such methods was posed as an important open question (see 
for example \citealp{nesterov2008accelerating}). The goal of this paper is to 
address this gap.

Specifically, we prove that when the dimension is sufficiently large, for 
the class of convex functions with $\mu_1$-Lipschitz gradients 
and $\mu_2$-Lipschitz Hessians, the worst-case oracle complexity of any 
deterministic algorithm is
\begin{equation}\label{eq:secondorder}
	\Omega\left(\min\left\{\sqrt{\frac{\mu_1 
			D^2}{\epsilon}}~,~\left(\frac{\mu_2D^3}{\epsilon}\right)^{2/7}\right\}\right).
\end{equation}
This bound is tight up to constants, as it is matched by a combination of 
existing methods in the literature (see discussion below). Moreover, if we 
restrict ourselves to functions which are $\lambda$-strongly 
convex, we prove an oracle complexity lower bound of
\begin{equation}\label{eq:secondorder_strong}
\Omega\left(\left(\min\left\{\sqrt{\frac{\mu_1}{\lambda}}~,~
\left(\frac{\mu_2}{\lambda}D\right)^{2/7}\right\}+
\log
\log_{18}\left(\frac{\lambda^3/\mu_2^2}{\epsilon}\right)\right)\right).
\end{equation}
Moreover, we establish that this bound is tight up to logarithmic factors 
(independent of $\epsilon$), 
utilizing a novel adaptation of the A-NPE algorithm proposed in 
\citet{monteiro2013accelerated} (see 
\appref{sec:opt_strongly_convex_alg}). These new lower 
bounds have several implications:
\begin{itemize}
\item Perhaps unexpectedly, 
\eqref{eq:secondorder_strong} 
establishes that one cannot 
avoid in general a polynomial dependence on geometry-dependent ``condition 
numbers'' of the form $\mu_1/\lambda$ or $\mu_2D/\lambda$, even with 
second-order methods. This is despite the ability of such methods to favorably 
alter the geometry of the problem (for example, the Newton method is 
well-known to be affine invariant).
\item To improve on the oracle complexity of 
first-order methods for strongly-convex problems 
(\eqref{eq:firstorder_strong}) by more than logarithmic factors, one cannot 
avoid a 
polynomial dependence on the initial distance 
$D$ to the optimum. This is despite the fact that the dependence on $D$ with 
first-order methods is only logarithmic. In fact, when $D$ is sufficiently 
large 
(of order $\frac{\mu_1^{7/4}}{\mu_2\lambda^{3/4}}$ or larger), second-order 
methods 
cannot improve on the oracle complexity of first-order methods by more than 
logarithmic factors. 
\item In the convex case, second-order methods are again no better than 
first-order methods in certain parameter regimes (i.e., when $\mu_2\geq 
\mu_1^{7/4}\sqrt{D}/\epsilon^{3/4}$), despite the availability of more 
information.
\end{itemize}

Finally, we show how our proof techniques can be generalized, to establish 
lower bounds for methods employing higher-order derivatives. In particular, for 
methods using all derivatives up to order $k$, we show that for convex 
functions with $\mu_k$-Lipschitz k-th order derivatives, the oracle complexity 
is 
\[
\Omega \left(\left(\frac{\mu_{k} D^{k+1}}{(k+1)!k\epsilon}\right)^{2/(3k+1)}\right)~.
\]
Note that this directly generalizes \eqref{eq:firstorder} for $k=1$, and 
\eqref{eq:secondorder} when $k=2$ and $\mu_1$ is unrestricted. 

\subsection*{Related Work}

Below, we review some pertinent results in the context of second-order methods. 
Related results in the contest of k-th order methods are discussed in 
Subsection \ref{subsec:higher}.

Perhaps the most well-known and fundamental second-order method is the Newton 
method, which relies on iterations of the form $\bw_{t+1}=\bw_{t}-
(\nabla^2 f(\bw))^{-1}\nabla f(\bw)$ 
(see e.g., \citet{boyd2004convex}). It is well-known that this method exhibits 
\emph{local} quadratic convergence, in the sense that if $f$ is strictly 
convex, and the method is initialized close enough to the optimum 
$\bw^*=\arg\min_{\bw} f(\bw)$, then 
$\Ocal(\log\log(1/\epsilon))$ iterations suffice 
to reach a solution $\bw$ such that $f(\bw)-f(\bw^*)\leq \epsilon$. However, 
in order to get global convergence
(starting from an arbitrary point not necessarily close to the optimum), one 
needs to make some algorithmic modifications, such as introducing a step size 
parameter or line search, employing trust region methods, or adding various 
types of regularization (see for example \citet{conn2000trust} and references 
therein). Despite the huge literature on the subject, the worst-case global 
convergence behavior of these methods is not well understood 
\citep{nesterov2006cubic}. 
For the Newton method with a line search, the number of 
iterations can be upper bounded by
\[
\Ocal\left(\frac{\mu_1^2 
\mu_2^2}{\lambda^5}(f(\bw_1)-f(\bw^*))+\log\log_2\left(\frac{\lambda^3/\mu_2^2}{\epsilon}\right)\right),
\]
where $\mu_1,\mu_2$ are the Lipschitz parameters of the gradients and Hessians 
respectively, and assuming the function is $\lambda$-strongly convex 
(\citet{kantorovich1948functional}, see also 
\cite{boyd2004convex}). Note that the first term
captures the initial phase required to get sufficiently close to $\bw^*$, whereas the second term captures the quadratically
convergent phase. Although the final convergence is rapid, the first phase is 
the dominant one in the bound (unless $\epsilon$ is exceedingly small). If $f$ 
is self-concordant\footnote{That is, for any vectors 
$\bv,\bw$, the function $g(t)=f(\bw+t\bv)$
satisfies $|g'''(t)|\leq 2g''(t)^{3/2}$}, this can be improved to 
\[
\Ocal\left((f(\bw_1)-f(\bw^*))+\log\log_2\left(\frac{1}{\epsilon}\right)\right),
\]
independent of the strong convexity and Lipschitz parameters 
(\citet{nesterov1994interior}). Unfortunately, 
not all practically relevant objective functions are self-concordant. For 
example, loss functions common in machine learning applications, such as the 
logistic loss $x\mapsto \log(1+\exp(-x))$, are not 
self-concordant\footnote{These can often be made self-concordant by re-scaling, 
smoothing and adding regularization (e.g. 
\citet{bach2010self}), but even when possible, these modifications strongly 
affect the 
$f(\bw_1)-f(\bw^*)$ 
term in the bound, and prevents it from being independent of the strong 
convexity and Lipschitz parameters.}, and our own results utilize the simple 
but not self-concordant function $x\mapsto |x|^3$. 

Returning to our setting of generic convex and smooth functions, and focusing 
on strongly convex functions for now, the best existing upper bounds (we are 
aware of) were obtained for cubic-regularized variants of the Newton method, 
where at each iteration one essentially minimizes a quadratic approximation of 
the function at the current point, regularized by a cubic term
\citep{nesterov2006cubic,nesterov2008accelerating}. The existing analysis (in 
section 6 of \citet{nesterov2008accelerating}) implies an oracle complexity 
bound of at most
\[
	\Ocal\left(\left(\frac{\mu_2 
	}{\lambda}D\right)^{1/3}+\log\log_2\left(\frac{\lambda^3/\mu_2^2}
	{\epsilon}\right)\right),
\]
where $D=\norm{\bw_1-\bw^*}$ is the distance from the initialization point 
$\bw_1$ to the optimum $\bw^*$ (see section 6 in 
\cite{nesterov2008accelerating}, as well as \cite{cartis2012evaluation} for 
another treatment of such cubic-regularized methods). However, as we show in 
\appref{sec:opt_strongly_convex_alg}, a better oracle complexity bound can be 
obtained, by adapting the A-NPE method proposed in 
\citep{monteiro2013accelerated} and 
analyzed for convex functions, to the strongly convex case. The resulting 
complexity upper bound is
\begin{equation}\label{eq:upsecond}
\Ocal\left(\left(\frac{\mu_2}{\lambda}D\right)^{2/7}\log\left(\frac{\mu_1\mu_2^2D^2}{\lambda^3}\right)+\log
\log_2\left(\frac{\lambda^3/\mu_2^2}{\epsilon}\right)\right).
\end{equation}

An alternative to the above is to use a hybrid scheme, starting with 
accelerated gradient descent (which is an optimal 
\emph{first-order} method for strongly convex functions with Lipschitz gradients) 
and when close enough to the optimal solution, switch to a cubic-regularized 
Newton method, which is quadratically converging in that 
region\footnote{Instead 
of cubic-regularized Newton, one can also use the standard Newton method, 
although the resulting bound using the existing analysis will have slightly 
worse logarithmic factors.}. The 
required number 
of iterations is then
\begin{equation}\label{eq:twophaseup}
\Ocal\left(\sqrt{\frac{\mu_1}{\lambda}}\cdot \log\left(\frac{\mu_1\mu_2^2
	D^2}{\lambda^3}\right)+\log\log_2\left(\frac{\lambda^3/\mu_2^2}
{\epsilon}\right)\right),
\end{equation}
where $D=\norm{\bw_1-\bw^*}$
(see \citet{nesterov2004introductory,nesterov2008accelerating}). Clearly, by 
taking the best of \eqref{eq:upsecond} and \eqref{eq:twophaseup} (depending on 
the parameters), one can theoretically attain an oracle complexity which is the 
minimum of \eqref{eq:upsecond} and \eqref{eq:twophaseup}. This minimum matches 
(up to a logarithmic factors) the lower bound in \eqref{eq:secondorder_strong}, 
which we establish in this paper. 

It is interesting to note that the 
bounds in \eqref{eq:upsecond} and \eqref{eq:twophaseup} are not directly 
comparable: The first bound has a polynomial dependence on $\mu_2/\lambda$ and 
$\norm{\bw_1-\bw^*}$, and a logarithmic dependence on $\mu_1$, whereas the 
second bound has a polynomial dependence on $\mu_1/\lambda$, logarithmic 
dependence on $\norm{\bw_1-\bw^*}$, and a logarithmic dependence on $\mu_2$. 
In a rather wide parameter regime (e.g. when $D$ is reasonably large, as often 
occurs in practice), the bound of the hybrid scheme can be better 
than that of pure second-order methods. In light of this, 
\citet{nesterov2008accelerating} raised the question of whether second-order 
schemes are indeed useful at the initial stage of the optimization process, for 
these types of problems. Our results indicate that indeed, in certain parameter 
regimes, this is not the case. 

Analogous results can be obtained for convex (not necessarily strongly convex) 
smooth functions. Using an appropriate analysis of the accelerated 
cubic-regularized Newton method \citep{nesterov2008accelerating}, one can 
attain a bound of
\[
\Ocal\left(\left(\frac{\mu_2D^3}{\epsilon}\right)^{1/3}\right)~.
\]
More recently, \citet{monteiro2013accelerated} proposed an accelerated hybrid 
proximal extragradient method, denoted as A-NPE, which attains a better bound 
of 
\begin{equation}\label{eq:anpe_convex}
\Ocal\left(\left(\frac{\mu_2D^3}{\epsilon}\right)^{2/7}\right).
\end{equation}
In addition, using an optimal first-order method (such as accelerated gradient 
descent), one can attain a bound of 
\begin{equation}\label{eq:twophaseup_convex}
\Ocal\left(\sqrt{\frac{\mu_1 D^2}{\epsilon}}\right)~.
\end{equation}
Clearly, by taking the best of the last two approaches (depending on the 
problem parameters), one can attain an oracle complexity equal to the minimum 
of the two bounds in \eqref{eq:anpe_convex} and \eqref{eq:twophaseup_convex}. 
This is matched (up to constants) by the lower bound in \eqref{eq:secondorder}, 
which we establish in this paper.

Finally, we discuss the few existing lower bounds known for 
second-order methods. If $\mu_2$ is not bounded (i.e.,the 
Hessians are not Lipschitz), it is easy to show that 
Hessian information is not useful. Specifically, the lower bound of 
\eqref{eq:firstorder} 
for first-order methods will then also apply to second-order methods, and in 
fact, to 
any method based on local information (see \citet[section 
7.2.6]{YudNem83} and \citet{arjevani2016oracle}). Of 
course, this lower bound does not apply to second-order methods when $\mu_2$ is 
bounded. In our setting, it is also possible to prove an 
$\Omega(\log\log(1/\epsilon))$ lower 
bound, even in one dimension \citep[section 8.1.1]{YudNem83}, but this does not 
capture the dependence on the strong convexity and Lipschitz parameters. 
Some algorithm-specific lower bounds in the context of non-convex 
optimization are provided in \citet{cartis2010complexity}. Finally, we were 
recently informed of a new work (\citet{agarwalhazan17}, yet unpublished at the 
time of writing), which 
uses a clean and elegant smoothing approach, to derive second- and
higher-order oracle lower bounds directly from known first-order oracle lower 
bounds, as well as extensions to randomized algorithms. However, the resulting 
bounds are not as tight as ours.

\section{Main Results}

In this section, we formally present our main results, starting with 
second-order oracle complexity bounds (\subsecref{subsec:secondorder}), and 
then discussing extensions to higher-order oracles (\subsecref{subsec:higher}).

\subsection{Second-order Oracle}\label{subsec:secondorder}

We consider a \emph{second-order} oracle, which given a point
$\bw$ returns the function's value $f(\bw)$, its gradient $\nabla f(\bw)$ and
its Hessian $\nabla^2 f(\bw)$, and algorithms, which produce a 
sequence of points $\bw_1,\bw_2,...,\bw_T$, with each $\bw_t$ being 
some deterministic function of the oracle's responses at 
$\bw_1,\ldots,\bw_{t-1}$. Our main results (for strongly convex and convex 
functions $f$ respectively) are provided below.

\begin{theorem}\label{thm:main_strong}
For any positive $\mu_1,\mu_2,\lambda,D,\epsilon$ such that 
$
\frac{\mu_1}{\lambda}\geq c_1, \frac{\mu_2}{\lambda}D\geq c_2$ and
$
\epsilon < \frac{c_3\lambda^3}{\mu_2^2}
$
(for some positive universal constants $c_1,c_2,c_3$), and 
any algorithm as above with 
initialization point $\bw_1$, 
there exists a function 
$f:\reals^d\rightarrow\reals$ (for some finite $d$) such that
\begin{itemize}
\item $f$ is
$\lambda$-strongly convex, twice-differentiable, has $\mu_1$-Lipschitz
gradients and $\mu_2$-Lipschitz Hessians, and has a global minimum $\bw^*$ 
satisfying $\norm{\bw_1-\bw^*}\leq D$.
\item The index $T$ required to ensure $f(\bw_T)-f(\bw^*) ~\leq~
\epsilon$ is at least
\[
c\cdot\left(\min\left\{\sqrt{\frac{\mu_1}{\lambda}}~,~
\left(\frac{\mu_2}{\lambda}D\right)^{2/7}\right\}+
\log
\log_{18}\left(\frac{\lambda^3/\mu_2^2}{\epsilon}\right)\right)
\]
for some universal constant $c>0$~.
\end{itemize}
\end{theorem}

\begin{theorem}\label{thm:main2}
	For any positive $\mu_1,\mu_2,D,\epsilon$ and any
	algorithm as above with initialization point $\bw_1$, there exists a 
	function $f:\reals^d \rightarrow 
	\reals$ (for some finite $d$) such that
	\begin{itemize}
		\item $f$ is convex, twice-differentiable, has $\mu_1$-Lipschitz 
		gradients and $\mu_2$-Lipschitz Hessians, and has a global minimum 
		$\bw^*$ 
		satisfying $\norm{\bw_1-\bw^*} \leq D$.
		\item The index $T$ required to 
		ensure $f(\bw_T)-f(\bw^*)\leq \epsilon$ is at least
		\[
		c\cdot \min\left\{\sqrt{\frac{\mu_1 
		D^2}{\epsilon}}~,~\left(\frac{\mu_2D^3}{\epsilon}\right)^{2/7}\right\}
		\]
		for some universal constant $c>0$~.
	\end{itemize}
\end{theorem}

We emphasize that the theorems focus on the high-dimensional setting, where the 
dimension $d$ is not necessarily fixed and may depend on other problem 
parameters.
Also, we note that the parameter 
constraints in \thmref{thm:main_strong} are purely for technical reasons (they 
imply that the different terms in the bound are at least some positive 
constant), and can probably be relaxed somewhat. 

Let us compare these theorems to the upper bounds discussed in the related work 
section, which are
\[
\Ocal\left(\min\left\{\sqrt{\frac{\mu_1}{\lambda}} 
,\left(\frac{\mu_2 
	D}{\lambda}\right)^{2/7}\right\}\cdot \log\left(\frac{\mu_1\mu_2^2
		D^2}{\lambda^3}\right)+\log\log_2\left(\frac{\lambda^3/\mu_2^2}
{\epsilon}\right)\right).
\]
in the strongly convex case,
and
\[
\Ocal\left(\min\left\{\sqrt{\frac{\mu_1 
		D^2}{\epsilon}}~,~\left(\frac{\mu_2D^3}{\epsilon}\right)^{2/7}\right\}\right).
\]
in the convex case. Our bound in the convex case is tight up to 
constants, and in the strongly convex case, up to a 
$\log(\mu_1\mu_2^2D^2/\lambda^3)$ factor. We conjecture that some such 
logarithmic factor 
(possibly a smaller one) is indeed necessary, in order to get a tight 
interpolation to the $\Omega(\sqrt{\mu_1/\lambda}\cdot\log(\mu_1 
D^2/\epsilon))$ lower bound of first-order methods as $\mu_2\rightarrow\infty$ 
(see \citet[section 7.2.6]{YudNem83} and \citet{arjevani2016oracle}), and that 
it can be recovered with a more careful analysis of our construction. However, 
this involves some non-trivial technical challenges,  which we leave 
to future work.

Comparing the lower and upper bounds in the strongly convex case, one can make 
the following observations: 
\begin{itemize}
\item The lower bound captures the two phases common in 
second-order methods such as the Newton method: An initial slow convergence 
from the initialization point to the local neighborhood of the optimum 
(captured by the $\min\left\{\sqrt{\frac{\mu_1}{\lambda}}~,~
\left(\frac{\mu_2}{\lambda}D\right)^{2/7}\right\}$ term), 
followed by a fast local quadratic convergence to the optimum (captured by the 
second term, which is doubly-logarithmic in the accuracy $\epsilon$). 
\item Unless $\epsilon$ is exceedingly small, the oracle complexity 
is dominated by the geometry-dependent terms
$\mu_1/\lambda$ and $\mu_2D/\lambda$. This is despite the fact that 
second-order methods can use Hessian information to alter the geometry of the 
problem (for example, the Newton method is well-known to be 
affine invariant). 
\item If $\mu_2D/\lambda$ is sufficiently large (specifically, if $D$ is order 
of $\frac{\mu_1^{7/4}}{\mu_2\lambda^{3/4}}$ or larger), then the 
lower bound becomes at least $\sqrt{\mu_1/\lambda}$, which is no betters what 
can be obtained with first-order methods up to logarithmic factors (see 
\eqref{eq:firstorder_strong}). Since $D$ 
often scales inversely with the strong convexity of the problem (e.g. since the 
strong convexity is due to a regularization term), this is a rather broad and 
reasonable regime. 
\item On the other hand, if $\mu_2D/\lambda$ is smaller than 
$\sqrt{\mu_1/\lambda}$, then 
the oracle complexity can be significantly better than that of first-order 
methods, but this still comes at the inevitable price of a polynomial 
dependence on the distance $D$ from the optimum. In contrast, first-order 
methods have only a logarithmic dependence on $D$ (see 
\eqref{eq:firstorder_strong}).
\end{itemize}

Similar types of conclusions regarding on the behavior of first and 
second-order methods can be drawn as in the strongly convex case.
Namely, if $\mu_2D^3/\epsilon$ 
is large enough (specifically, if $\mu_2\geq 
\mu_1^{7/4}\sqrt{D}/\epsilon^{3/4}$), the 
complexity of second-order methods is not significantly better than what can be 
obtained with first-order methods. 

\subsection{Higher Order Oracles}\label{subsec:higher}

In addition to first-order and second-order oracles, it is of interest to 
understand what can be achieved with methods employing higher order 
derivatives. It turns out that the techniques we use to establish our 
second-order lower bounds can be easily generalized to such higher-order 
methods. 

More explicitly, we consider methods which can be modelled as interacting with 
a \emph{k-th order} oracle, which given a point
$\bw$ returns the function's value and all of its derivatives up to order $k$, 
namely,  $f(\bw),\nabla f(\bw),\nabla^2 f(\bw),\ldots,\nabla^k f(\bw)$. Given 
access to such an oracle, the method produces a sequence of points 
$\bw_1,\bw_2,\ldots,\bw_T$ as before (where each $\bw_t$ is a deterministic 
function of the previous oracle responses). For simplicity, we will focus here 
on the 
case of convex functions (not necessarily strongly convex), where the $k$-th 
order derivative is Lipschitz continuous.

\begin{theorem}\label{thm:main3}
	For any positive integer $k$, positive $\mu_k,D,\epsilon$, and algorithm 
	based on a $k$-th order oracle as above, there 
	exists a function $f:\reals^d \rightarrow 
	\reals$ (for some finite $d$) such that
	\begin{itemize}
		\item $f$ is convex, $k$ times differentiable, $k$-order smooth (i.e., 
		$\norm{\nabla^kf(\bu)-\nabla^kf(\bv)} \leq \mu_k \norm{\bu -\bv}$) and 
		has a global minimum 
		$\bw^*$ 
		satisfying $\norm{\bw_1-\bw^*} \leq D$.
		\item The index $T$ required to 
		ensure $f(\bw_T)-f(\bw^*)\leq \epsilon$ is at least
		\[
		c\left(\frac{\mu_{k} D^{k+1}}{(k+1)!k\epsilon}\right)^{2/(3k+1)}~,
		\]
		for some universal constant $c>0$~.
	\end{itemize}
\end{theorem}
Note that this result directly generalizes existing results for first-order 
oracles ($k=1$), as well as our results for second-order oracles ($k=2$, when 
$\mu_1$ is unrestricted). 

Finally, we compare our lower bound to the best upper bound we are aware of, 
established by \cite{baes2009estimate} using a high-order method with oracle 
complexity of
\[
\Ocal\left(\sqrt{k}\left(\frac{f(\bw_1)-f(\bw^*)}{\epsilon}+\frac{\mu_kD^{k+1}}{(k+1)!\epsilon}\right)^{1/(k+1)}\right)~.
\]
Note that the upper bound contains an additional $(f(\bw_1)-f(\bw^*))/\epsilon$ 
term, and moreover, the exponent (as a function of $k$) is larger than ours 
($1/(k+1)$ vs. $2/(3k+1)$). Based on our results, we know that this upper bound 
is loose in the $k=2$ case, so we conjecture that it is indeed loose for all 
$k$, and can be improved.

\section{Proof Ideas}

The proofs of our theorems are based on a careful modification of a 
standard lower bound construction for first-order methods (see 
\cite{nesterov2004introductory}). That construction uses quadratic functions, 
which in the convex case and ignoring various 
parameters, have a basic structure of the form
\[
f_T(\bw) = f_T(w_1,w_2,\ldots) = w_1^2+\sum_{j=1}^{T-1}(w_j-w_{j+1})^2+w_T^2-w_1
\]
(more precisely, one considers $f_T(V\bw)$ for a certain orthogonal matrix $V$, 
and use additional parameters depending on the smoothness). A crucial 
ingredient of 
the proof is that the 
function $x\mapsto x^2$ has a value and derivative of zero at the origin, which 
allows us to construct a function which ``hides'' information from an 
algorithm relying solely on values and gradients. This can be shown to lead to 
an optimization error lower bound of the form $\min_{\bw} f_{T}(\bw)-\min_{\bw} 
f_{2T}(\bw)$ after $T$ oracle queries, which for first-order methods leads to 
an $\Omega(\mu_1 D^2/T^2)$ lower bound on the error, translating to an 
$\Omega(\sqrt{\mu_1 D^2/\epsilon})$ lower bound on $T$. However, this 
construction leads to trivial bounds for second-order methods, since 
given the Hessian and a gradient of a quadratic function at just a single 
point, one can already compute the exact minimizer.

Our approach to handle second-order (and more generally, $k$-order) methods is 
quite simple: Instead of 
$x\mapsto x^2$, we 
rely on mappings of the form $x\mapsto |x|^{k+1}$, and use functions with the 
basic structure
\[
f_T(\bw) = 
|w_1|^{k+1}+\sum_{j=1}^{T-1}|w_j-w_{j+1}|^{k+1}+|w_T|^{k+1}-w_1.
\]
The intuition is that $x\mapsto |x|^{k+1}$ has a value
\emph{and first k derivatives} of zero at the origin, and therefore variants of 
the function above can be used to ``hide'' information from the algorithm, even 
if it can receive Hessians or higher-order derivatives of the function. Another 
motivation for choosing such functions is that they are generally not 
self-concordant, and therefore the upper bounds relevant to self-concordant 
functions do not apply. We rely on this construction and arguments similar to 
those of first-order oracle lower bounds, to get our results. 

In the derivation of our results for second-order methods, there are two 
technical challenges that need to be overcome: The first is that $f_T$, as 
defined above (for $k=2$), can be shown to have globally Lipschitz Hessians, 
but not globally Lipschitz gradients as required by our theorems. To tackle 
this, we replace the mapping $x\mapsto |x|^3$ by a more complicated mapping, 
which is cubic close to the origin and quadratic further away. This necessarily 
complicates the proof. The second challenge is that due to the cubic terms, 
computing the minimizer of $f_T$ and its minimal value is more challenging than 
in first-order lower bounds, especially in the strongly convex case (where we 
are unable to even find a closed-form expression for the minimizer, and resort 
to bounds instead). Again, this makes the analysis more complicated.

We conclude this section by sketching how our bounds can be derived in case of 
second-order methods, and in the simplest possible setting, where we wish to 
obtain an $\Omega((D^3/\epsilon)^{2/7})$ lower bound for 
the class of convex functions with Lipschitz Hessians (and no 
assumptions on the Lipschitz parameter of the gradients), assuming the 
algorithm makes its first query at the origin. In that case, consider the 
function $f_T$ in 
this class of the form
\[
f_T(\bw) = |w_1|^3+\sum_{j=1}^{T-1}|w_j-w_{j+1}|^3+|w_T|^3-3\gamma\cdot w_1,
\]
where $\gamma$ is a parameter to be chosen later. Computing the derivatives and 
setting to zero, and arguing that the minimizer must have non-negative 
coordinates, we get that the optimum satisfies
\[
w_1^2+(w_1-w_2)^2 = 
\gamma~~,~~w_T^2=(w_{T-1}-w_T)^2
\]
and
\[
\forall j=2,3,\ldots,T-1~,~~~(w_{j-1}-w_j)^2=(w_j-w_{j+1})^2.
\]
It can be verified that this is satisfied by $w_j = 
(T+1-j)\sqrt{\frac{\gamma}{T^2+1}}$ for all $j=1,2,\ldots,T$, and that this is 
the unique 
minimizer of $f_T$ as a function on $\reals^T$. Moreover, 
assuming $\gamma\leq D^2/T$, the norm of this minimizer (and hence the initial 
distance to it from the algorithm's first query point, by assumption) is at 
most $D$ as required. Plugging this $\bw$ into $f_T$, we get that 
$\min_{\bw} f_T(\bw)$ equals
\begin{align*}
f_T(\bw) &= 
T^3\left(\frac{\gamma}{T^2+1}\right)^{3/2}+(T-1)\left(\frac{\gamma}{T^2+1}\right)^{3/2}
+ \left(\frac{\gamma}{T^2+1}\right)^{3/2}-3\gamma T\sqrt{\frac{\gamma}{T^2+1}}\\
&=T(T^2+1)\left(\frac{\gamma}{T^2+1}\right)^{3/2}-
3\frac{\gamma^{3/2}T}{\sqrt{T^2+1}}~=~ -\frac{2\gamma^{3/2}T}{\sqrt{T^2+1}}~.
\end{align*}
Now, using arguments very similar to those in first-order oracle 
complexity lower bounds \citep{nesterov2004introductory}, it is possible 
to 
construct a function for which the 
optimization error of the algorithm is lower bounded by 
$\min_{\bw}f_{T}(\bw)-\min_{\bw}f_{2T}(\bw)$. By the calculations above, this 
in turn equals
\[
2\gamma^{3/2}\left(\frac{2T}{\sqrt{4T^2+1}}
-\frac{T}{\sqrt{T^2+1}}\right)
~=~2\gamma^{3/2}\left(\frac{1}{\sqrt{1+\frac{1}{4T^2}}}-\frac{1}{\sqrt{1+\frac{1}{T^2}}}\right).
\]
Using the fact that $\frac{1}{\sqrt{1+x}}\approx 1-\frac{1}{2}x$ for small $x$, 
this 
equals $\Omega(\gamma^{3/2}/T^2)$. Choosing $\gamma$ on the order of $D^2/T$ 
(as required earlier to satisfy the norm constraint on the minimizer), we get 
a lower bound of 
$\Omega(D^3/T^{7/2})$ on the 
optimization error 
$\epsilon$, or equivalently, a lower bound of $\Omega((D^3/\epsilon)^{2/7})$ on 
$T$.

\section{Proof of \thmref{thm:main_strong}}\label{sec:proof_main_strong}

We will assume without loss of generality that the algorithm initializes at 
$\bw_1=\mathbf{0}$ (if that is not the case, one can simply replace the 
``hard'' function $f(\bw)$ below by $f(\bw-\bw_1)$, and the same proof holds 
verbatim). Thus, the theorem requires that our function has a minimizer $\bw^*$ 
satisfying $\norm{\bw^*}\leq D$. 

Let $\Delta,\gamma$ be parameters to be chosen later. Define 
$g:\reals\mapsto\reals$ as
\[
g(x) = \begin{cases} \frac{1}{3}|x|^3 & |x|\leq \Delta\\ 
\Delta x^2-\Delta^2|x|+\frac{1}{3}\Delta^3&|x|>\Delta ,
\end{cases}.
\]
which is easily verified to be convex and twice continuously differentiable, 
and let
$\bv_1,\bv_2,\ldots,\bv_{\tilde{T}}$ be orthogonal unit vectors in $\reals^d$ 
which will be
specified
later. Letting the number of iterations $T$ be fixed, we consider the function
\[
f(\bw) =
\frac{\mu_2}{12}\left(\sum_{i=1}^{\tilde{T}-1}g(\inner{\bv_i,\bw}
-\inner{\bv_{i+1},\bw})-\gamma\inner{\bv_1,\bw}\right)
+\frac{\lambda}{2}\norm{\bw}^2~,
\]
where $\tilde{T}\geq 
\max\left\{4\gamma\left(\frac{\mu_2}{6\lambda}\right)^2+1,
2T,\frac{\gamma\mu_2}{6\lambda}+1\right\}$ is 
some sufficiently large number, and the dimension $d$ is at least $2\tilde{T}$. 

The proof is constructed of several parts: First, we analyze properties of the 
global minimum of $f$ (Subsection \ref{subsec:minimizer}). Then, we prove the 
oracle 
complexity lower bound in Subsection \ref{subsec:mainproof} (depending on 
$\Delta,\gamma$), and finally, in Subsection 
\ref{subsec:smoothstrong}, we choose the parameters so that $f$ indeed has
the various geometric properties specified in the theorem. 

\subsection{Minimizer of $f$}\label{subsec:minimizer}

The goal of this subsection is to prove the following proposition, which 
characterizes key properties of the global minimum of 
$f$:
\begin{proposition}\label{prop:opt} Suppose that 
$ \gamma\geq 10^4\left(\frac{\lambda}{\mu_2}\right)^2$ and $\Delta\geq 
\sqrt{\gamma}$. 
Then $f$ has a unique minimizer $\bw^*$ which satisfies the following:
	\begin{enumerate}
		\item For any $t\in \{1,2,\ldots,\tilde{T}\}$, it holds that  
		$\inner{\bv_t,\bw^*} \geq
		\max\left\{0~,~\frac{\gamma^{3/4}}{7\sqrt{12\lambda/\mu_2}}
			+\sqrt{\gamma}\left(\frac{1}{2}-t\right)\right\}$.
		\item There exists some $t_0 \leq \tilde{T}/2$ such that for all 
		indices $k\in \{0,1,\ldots,\tilde{T}-t_0\}$, it holds that\\
		$\inner{\bv_{t_0+k},\bw^*} ~\geq~ 
		\frac{108\lambda}{\mu_2}\cdot(18)^{-2^k}$.
		\item $\norm{\bw^*}^2 \leq 
		\frac{2\gamma^{7/4}}{(12\lambda/\mu_2)^{3/2}}$~.
	\end{enumerate}
\end{proposition}

Since $f$ is strongly convex, its global minimizer is unique and well-defined.
To prove the proposition, we will consider the simpler strongly-convex function
\begin{equation}\label{eq:tildef}
\tilde{f}(\bw) =
\frac{1}{3}\sum_{i=1}^{\tilde{T}-1}|w_i-w_{i+1}|^3
+\frac{\tilde{\lambda}}{2}\norm{\bw}^2-\gamma\cdot w_1,
\end{equation}
where 
\[
\tilde{\lambda}:=\frac{12\lambda}{\mu_2}~,
\]
and prove that its minimizer $\tilde{\bw}^*$ satisfies the 
following:
\begin{enumerate}
	\item For any $t\in \{1,2,\ldots,\tilde{T}\}$, it holds that   
	$\tilde{w}^*_t \geq
	\max\left\{0~,~\frac{\gamma^{3/4}}{7\sqrt{\lambda}}
	+\sqrt{\gamma}\left(\frac{1}{2}-t\right)\right\}$
	(\lemref{lem:phase1}).
	\item There exists some  $t_0 \leq \tilde{T}/2$ such that 
	for all $k\in \{0,1,\ldots,\tilde{T}-t_0\}$, it holds that
	$\tilde{w}^*_{t_0+k} ~\geq~ 9\tilde{\lambda}\cdot(18)^{-2^k}$ 
	(\lemref{lem:phase2}).
	\item $\sum_{i=1}^{\tilde{T}}\tilde{w}^{*2}_{i} \leq 
	\frac{2\gamma^{7/4}}{\tilde{\lambda}^{3/2}}$ (\lemref{lem:phase3})
\end{enumerate}
We then argue  that the minimizer $\bw^*$ of $f$ satisfies
$\inner{\bv_i,\bw^*}=\tilde{w}^*_{i}$ for all $i=1,2,\ldots,\tilde{T}$ (\lemref{lem:phase4}), and 
that $\norm{\bw^*}^2 = \sum_{i=1}^{\tilde{T}}\inner{\bv_i,\bw^*}^2$ 
(\lemref{lem:phase5}), from which Proposition
\ref{prop:opt} follows.

We begin with the following technical key result:
\begin{lemma}\label{lem:woptstruct}
It holds that $\tilde{w}^*_1\geq \tilde{w}^*_2\geq
\cdots \geq \tilde{w}^*_{\tilde{T}} \geq  0$, and
\[
\tilde{w}^*_{t+1} = w_{t}^*-\sqrt{\gamma-\tilde{\lambda}
\sum_{j=1}^{t}\tilde{w}^*_j}~~~~~~~\forall t\in \{1,2,\ldots,\tilde{T}-1\}~.
\]
Moreover, 
$\sum_{j=1}^{\tilde{T}}\tilde{w}^*_j = \frac{\gamma}{\tilde{\lambda}}$.
\end{lemma}
\begin{proof}
We begin by showing that $\tilde{w}^*_j\geq 0$ for all $j$, first for $j=1$ and 
then
for $j>1$. To do so, note that $\tilde{f}(\mathbf{0})=0$ yet $\nabla
\tilde{f}(\mathbf{0})=-\gamma\cdot \be_1\neq
\mathbf{0}$, and therefore $\mathbf{0}$ is a sub-optimal point. Thus, we must
have $\tilde{f}(\tilde{\bw}^*)<0$. The only negative term in the definition of 
$\tilde{f}(\cdot)$ is
$-\gamma\cdot w_1$, so we must have $\tilde{w}^*_1>0$. We now argue that 
$w_j\geq 0$ 
for 
all $j>1$:
Otherwise, let $\bw$ be the vector which equals $w_j=|\tilde{w}^*_j|$ for all 
$j$, and
note that $w_1=\tilde{w}^*_1$ since we just showed $\tilde{w}^*_1> 0$. Based on 
this, it is
easily verified
that
\[
\tilde{f}(\bw)-\tilde{f}(\tilde{\bw}^*) =
\frac{1}{3}\sum_{i=1}^{\tilde{T}-1}\left(\big||\tilde{w}^*_i|-|\tilde{w}^*_{i+1}|\big|^3-|\tilde{w}^*_i-\tilde{w}^*_{i+1}|^3
\right)~\leq~0~,
\]
which means that $\bw$ is the (unique) minimum of $\tilde{f}$, hence 
$\bw=\tilde{\bw}^*$. By
definition of $\bw$, this implies $\tilde{w}^*_j=|\tilde{w}^*_j|$ for all $j$, 
hence $\tilde{w}^*_j\geq
0$ for all $j$.

We now turn to prove that $\tilde{w}^*_j$ is monotonically decreasing in $j$. 
Suppose
on the contrary that this is not the case, and let $j_0$ be the smallest index
for which $\tilde{w}^*_{j_0}<\tilde{w}^*_{j_0+1}$, and let $\delta := 
\tilde{w}^*_{j_0+1}-\tilde{w}^*_{j_0}>0$.
Define the
vector $\bw$ to be
\[
w_i = \begin{cases} \tilde{w}^*_i & i\leq j_0 \\ \max\{0,\tilde{w}^*_i-\delta\} 
& d \geq i>j_0
\end{cases}.
\]
Note that this vector must be different than $\bw$, as
$w_{j_0+1}=\max\{0,\tilde{w}^*_{j_0+1}-\delta\}=\max\{0,\tilde{w}^*_{j_0}\}=\tilde{w}^*_{j_0}=w_{j_0}$,
hence $w_{j_0+1}=w_{j_0}$ yet $\tilde{w}^*_{j_0+1}>\tilde{w}^*_{j_0}$ by 
assumption. On the
other hand, it is easily verified that $|w_i-w_{i+1}|^3\leq
|\tilde{w}^*_i-\tilde{w}^*_{i+1}|^3$ and $w_i^2 \leq (\tilde{w}^*_i)^2$ for
all\footnote{This is trivially true for $i<j_0$. For $i=j_0$, we have
$|w_{j_0}-w_{j_0+1}|^3=0< |\tilde{w}^*_{j_0}-\tilde{w}^*_{j_0+1}|^3$ and
$w_{j_0}^2=(\tilde{w}^*_{j_0})^2$. For $i>j_0$, we have $|w_i-w_{i+1}|^3=
|\max\{0,\tilde{w}^*_i-\Delta\}-\max\{0,\tilde{w}^*_{i+1}-\Delta\}|^3\leq
|(\tilde{w}^*_i-\Delta)-(\tilde{w}^*_{i+1}-\Delta)|^3 = 
|\tilde{w}^*_i-\tilde{w}^*_{i+1}|^3$, and moreover,
$w_i^2 = \max\{0,\tilde{w}^*_i-\Delta\}^2$, which is $0$ (hence $\leq 
(\tilde{w}^*_i)^2$) if
$\tilde{w}^*_i\leq \Delta$ and less
than $(\tilde{w}^*_i)^2$ if $\tilde{w}^*_i>\Delta$.} $i$,
and therefore $\tilde{f}(\bw)\leq \tilde{f}(\tilde{\bw}^*)$. But since 
$\tilde{\bw}^*$ is the
unique global minimizer and $\bw\neq \tilde{\bw}^*$, we get a contradiction, so 
we must
have $\tilde{w}^*_j$ monotonically decreasing for all $j$.

We now turn to prove the recursive relation $\tilde{w}^*_{t+1} =
w_{t}^*-\sqrt{\gamma-\tilde{\lambda} \sum_{j=1}^{t}\tilde{w}^*_j}$. By
differentiating $\tilde{f}$ and
setting to zero (and using the fact that $\tilde{w}^*_j$ is monotonically 
decreasing in
$j$), we get that
\begin{equation}\label{eq:edge}
(\tilde{w}^*_1 - \tilde{w}^*_2)^2 = \gamma-\tilde{\lambda} 
w_1^*~~~~,~~~~(\tilde{w}^*_{\tilde{T}-1} - \tilde{w}^*_{\tilde{T}})^2 = 
\tilde{\lambda} w_{\tilde{T}}^*
\end{equation}
and 
\begin{equation}\label{eq:diff}
(\tilde{w}^*_{t}-\tilde{w}^*_{t+1})^2 =
(\tilde{w}^*_{t-1}-\tilde{w}^*_{t})^2-\tilde{\lambda}
\tilde{w}^*_{t}~~~~\forall t\in \{2,3,\ldots,\tilde{T}-1\}~.
\end{equation}
By unrolling this recursive form, we get
\[
(\tilde{w}^*_{t}-\tilde{w}^*_{t+1})^2 =
\gamma-\tilde{\lambda}\sum_{j=1}^{t}\tilde{w}^*_j~~~~\forall 
t\in\{1,2,\ldots,\tilde{T}-1\}~,
\]
from which the equation
\begin{equation}\label{eq:recurse}
\tilde{w}^*_{t+1} =
w_{t}^*-\sqrt{\gamma-\tilde{\lambda} \sum_{j=1}^{t}\tilde{w}^*_j}~~~~~~ \forall 
t\in\{1,2,\ldots,\tilde{T}-1\}
\end{equation}
follows, again using the
monotonicity of $\tilde{w}^*_t$ in $t$.

It remains to prove that  $\sum_{j=1}^{\tilde{T}}\tilde{w}^*_j = 
\frac{\gamma}{\tilde{\lambda}}$. By summing both sides of  \eqref{eq:diff} from 
$t=2$ to $t=\tilde{T}-1$ we have that:
\[
(\tilde{w}^*_{\tilde{T}-1}-\tilde{w}^*_{\tilde{T}})^2 =
(\tilde{w}^*_{1}-\tilde{w}^*_{2})^2-\tilde{\lambda}\sum_{t=2}^{\tilde{T}-1}\tilde{w}^*_{t}
\]
So by using \eqref{eq:edge} we get the desired equality.
\end{proof}

\begin{lemma}\label{lem:phase1}
For any $t\in\{1,2,\ldots,\tilde{T}\}$,
\[
\tilde{w}^*_t \geq
\max\left\{0~,~\frac{\gamma^{3/4}}{7\sqrt{\tilde{\lambda}}}+\sqrt{\gamma}\left(\frac{1}{2}-t\right)\right\}~.
\]
\end{lemma}
\begin{proof}
By the displayed equation in \lemref{lem:woptstruct}, we clearly have
$\tilde{w}^*_{t+1}\geq \tilde{w}^*_t-\sqrt{\gamma}$ for all $t\leq 
\tilde{T}-1$, and therefore
\begin{equation}\label{eq:lowb0}
\tilde{w}^*_t\geq \tilde{w}^*_1-(t-1)\sqrt{\gamma}~~~~\forall t\in 
\{1,2,\ldots,\tilde{T}\}~.
\end{equation}

Using the facts that $\tilde{w}^*_t$
is also always non-negative, that $\tilde{T}\geq 
\frac{\gamma\mu_2}{6\lambda}+1=\frac{\gamma}{\tilde{\lambda}}+1 \geq 
\tilde{w}^*_1 
+ 1$, and by \lemref{lem:woptstruct},
\begin{align*}
\frac{\gamma}{\tilde{\lambda}}&~=~\sum_{t=1}^{\tilde{T}}\tilde{w}^*_t ~\geq~
\sum_{t=1}^{\tilde{T}}\max\{0,\tilde{w}^*_1-(t-1)\sqrt{\gamma}\}~=~ 
\sum_{t=1}^{\lfloor
\tilde{w}^*_1/\sqrt{\gamma}+1\rfloor}\left(\tilde{w}^*_1-(t-1)\sqrt{\gamma}\right)\\
&~=~ \left\lfloor \frac{\tilde{w}^*_1}{\sqrt{\gamma}}+1\right\rfloor\cdot 
\tilde{w}_1^*-\sqrt{\gamma}\frac{\left(\left\lfloor 
\frac{\tilde{w}^*_1}{\sqrt{\gamma}}+1\right\rfloor-1\right)\left\lfloor 
\frac{\tilde{w}^*_1}{\sqrt{\gamma}}+1\right\rfloor}{2}
~\geq~ 
\frac{(\tilde{w}^*_1)^2}{\sqrt{\gamma}}-\sqrt{\gamma}
\frac{\frac{\tilde{w}^*_1}{\sqrt{\gamma}}\left(\frac{w_1^*}{\sqrt{\gamma}}+1\right)}{2},
\end{align*}
which implies that 
$(\tilde{w}^*_1)^2-\sqrt{\gamma}\cdot 
w_1^*-\frac{2\gamma^{3/2}}{\tilde{\lambda}}\leq
0$, which implies in turn
\begin{equation}\label{eq:w1upbound}
\tilde{w}^*_1\leq 
\frac{\sqrt{\gamma}+\sqrt{\gamma+8\gamma^{3/2}/\tilde{\lambda}}}{2}\leq
\frac{\sqrt{\gamma}+\sqrt{\gamma}+\sqrt{8\gamma^{3/2}/\tilde{\lambda}}}{2} =
\sqrt{\gamma}+\sqrt{\frac{2\gamma^{3/2}}{\tilde{\lambda}}}.
\end{equation}
On the other hand, again by \lemref{lem:woptstruct},
we know
that
\[
\tilde{w}^*_{t+1}\leq \tilde{w}^*_{t}-\frac{\sqrt{\gamma}}{2}~,~~~~~\forall 
t\in\{1,2,\ldots,\tilde{T}-1\}:\sum_{j=1}^{t}\tilde{w}^*_j\leq
\frac{3\gamma}{4\tilde{\lambda}},
\]
and hence
\begin{equation}\label{eq:lowb1}
\tilde{w}^*_{t+1}\leq \tilde{w}^*_1-\frac{t\sqrt{\gamma}}{2}~,~~~~~\forall 
t\in\{1,2,\ldots,\tilde{T}-1\}:\sum_{j=1}^{t}\tilde{w}^*_j\leq
\frac{3\gamma}{4\tilde{\lambda}}.
\end{equation}
Let $t_0\in \{1,2,\ldots,\tilde{T}\}$ be the smallest index such that 
$\sum_{j=1}^{t_0}\tilde{w}^*_j > 
\frac{3\gamma}{4\tilde{\lambda}}$ (such an index
must exist since 
$\sum_{j=1}^{\tilde{T}}\tilde{w}^*_j=\frac{\gamma}{\tilde{\lambda}}$).
Since $\frac{3\gamma}{4\tilde{\lambda}}<\sum_{j=1}^{t_0}\tilde{w}^*_j\leq t_0
\tilde{w}^*_1\leq
t_0\left(\sqrt{\gamma}+\sqrt{2\gamma^{3/2}/\tilde{\lambda}}\right)$ by 
\eqref{eq:w1upbound}, it follows
that
\[
t_0 ~\geq~ 
\frac{3\gamma}{4\tilde{\lambda}\left(\sqrt{\gamma}+\sqrt{2\gamma^{3/2}/\tilde{\lambda}}\right)}~=~
\frac{3\sqrt{\gamma}}{4\left(\tilde{\lambda}+\sqrt{2\gamma^{1/2}\tilde{\lambda}}\right)}~.
\]
According to \eqref{eq:lowb1} and the fact that $\tilde{w}^*_{t_0}\geq 0$, 
it
follows that
\[
0\leq \tilde{w}^*_{t_0}\leq \tilde{w}^*_1-\frac{(t_0-1)\sqrt{\gamma}}{2},
\]
and hence
\[
\tilde{w}^*_1 ~\geq~ \frac{(t_0-1)\sqrt{\gamma}}{2}~\geq~
\frac{3\gamma}{8(\tilde{\lambda}+\sqrt{2\gamma^{1/2}\tilde{\lambda}})}-\frac{\sqrt{\gamma}}{2}.
\]
Using this and \eqref{eq:lowb0}, it follows that for all $t\leq \tilde{T}$,
\[
\tilde{w}^*_t \geq
\frac{3\gamma}{8(\tilde{\lambda}+\sqrt{2\gamma^{1/2}\tilde{\lambda}})}+\sqrt{\gamma}\left(\frac{1}{2}-t\right)~.
\]
Since we assumed $\gamma\geq 10^4(\lambda/\mu_2)^2 > (12\lambda/\mu_2)^2 = 
\tilde{\lambda}^2$, we have $\tilde{\lambda}< 
\sqrt{\gamma^{1/2}\tilde{\lambda}}$, so the above can be lower bounded by the 
simpler expression $\gamma^{3/4}/7\sqrt{\tilde{\lambda}}+\sqrt{\gamma}(1/2-t)$. 
Since we also know that $\tilde{w}^*_t$ is non-negative, the result follows.

\end{proof}

\begin{lemma}\label{lem:phase2}
There exists  an index $t_0 \leq \tilde{T}/2$ such that
\[
\tilde{w}^*_{t_0+k} ~\geq~ 9\tilde{\lambda}\cdot(18)^{-2^k}~~~~\forall k\in 
\{0,1,\ldots,\tilde{T}-t_0\}
\]
\end{lemma}
\begin{proof}
By \lemref{lem:woptstruct}, it holds for any $t\in \{1,2,\ldots,\tilde{T}-1\}$ 
that
\begin{equation}\label{eq:reverserec}
\tilde{w}^*_{t}~=~\tilde{w}^*_{t+1}+\sqrt{\gamma-\tilde{\lambda}\sum_{j=1}^{t}\tilde{w}^*_j}
~=~\tilde{w}^*_{t+1}+\sqrt{\gamma-\tilde{\lambda}\left(\frac{\gamma}
{\tilde{\lambda}}-\sum_{j=t+1}^{\tilde{T}}\tilde{w}^*_j\right)}
~=~
\tilde{w}^*_{t+1}+\sqrt{\tilde{\lambda}\sum_{j=t+1}^{\tilde{T}}\tilde{w}^*_j}~.
\end{equation}
In particular, since $\tilde{w}^*_{j}\geq 0$ for all $j\leq \tilde{T}$, it 
follows that 
$\tilde{w}^*_t\geq
\sqrt{\tilde{\lambda}\sum_{j=t+1}^{\tilde{T}}\tilde{w}^*_j}\geq
\sqrt{\tilde{\lambda} \tilde{w}^*_{t+1}}$, and
therefore
\begin{equation}\label{eq:quadup}
\tilde{w}^*_{t+1}~\leq~ \frac{1}{\tilde{\lambda}}(\tilde{w}^*_t)^2 ~~~~\forall
t\in \{1,2,\ldots,\tilde{T}-1\}~.
\end{equation}
Let $t\leq \tilde{T}-1$ be any index such that\footnote{Such an index must 
exist: By assumption, 
$\tilde{T}\geq 
2\gamma\left(\frac{\mu_2}{6\lambda}\right)^2=\frac{2\gamma}{\tilde{\lambda}^2}$,
 so by
\lemref{lem:woptstruct}, $\frac{\gamma}{\tilde{\lambda}} 
 =\sum_{t=1}^{\tilde{T}}\tilde{w}^*_t \geq \tilde{T}\tilde{w}^*_{\tilde{T}} 
 \geq \frac{2\gamma}{\tilde{\lambda}^2}\tilde{w}^*_{\tilde{T}}$, hence 
 $\tilde{w}_{\tilde{T}}\leq \tilde{\lambda}/2$.}
$\tilde{w}^*_{t+1} \leq 
\frac{\tilde{\lambda}}{2}$.
By
\eqref{eq:quadup}, this implies that
\begin{align*}
\sum_{j=t+1}^{\tilde{T}}\tilde{w}^*_j ~\leq~
\sum_{k=0}^{\tilde{T}-t-1}\tilde{\lambda}\left(\frac{\tilde{w}^*_{t+1}}{\tilde{\lambda}}\right)^{2^{k}}~=~
\sum_{k=0}^{\tilde{T}-t-1}\tilde{w}^*_{t+1}\left(\frac{\tilde{w}^*_{t+1}}{\tilde{\lambda}}\right)^{2^{k}-1}~\leq~
\sum_{k=0}^{\tilde{T}-t-1}\tilde{w}^*_{t+1} \left(\frac{1}{2}\right)^{2^k-1} ~<~
 2\cdot \tilde{w}^*_{t+1}.
\end{align*}
Using the inequality above together with
\eqref{eq:reverserec} and the monotonicity of $\tilde{w}^*_t$, we get that for 
all
$t\leq \tilde{T}-1$ such that $\tilde{w}^*_{t+1}\leq\frac{\tilde{\lambda}}{2}$,
\begin{align*}
\tilde{w}^*_{t}
&~=~\tilde{w}^*_{t+1}+\sqrt{\tilde{\lambda}\sum_{j=t+1}^{\tilde{T}}\tilde{w}^*_j}
~\leq~\tilde{w}^*_{t+1}+\sqrt{2\tilde{\lambda} \tilde{w}^*_{t+1}}
~=~\sqrt{\tilde{w}^*_{t+1}}\left(\sqrt{\tilde{w}^*_{t+1}}+\sqrt{2\tilde{\lambda}}
\right)\\
&~\leq~\sqrt{\tilde{w}^*_{t+1}}\left(\sqrt{\frac{\tilde{\lambda}}{2}}+\sqrt{2\tilde{\lambda}}\right)
~\leq~ 3\sqrt{\tilde{\lambda} \tilde{w}^*_{t+1}}.
\end{align*}
This chain of inequalities implies that
\[
w_{t+1}\geq
\frac{(\tilde{w}^*_t)^2}{9\tilde{\lambda}}~~~~~~\forall t\in 
\{1,2,\ldots,\tilde{T}-1\}:
\tilde{w}^*_{t+1}\leq\frac{\tilde{\lambda}}{2}.
\]
Let $t_0\leq \tilde{T}/2$ denote the unique index that satisfies
$\tilde{w}^*_{t_0}>\frac{\tilde{\lambda}}{2}$, as well as
$\tilde{w}^*_{t_0+1}\leq
\frac{\tilde{\lambda}}{2}$ for all $t$ between $t_0$ and 
$\tilde{T}-1$~\footnote{Since $\tilde{w}^*_t$ monotonically decrease in $t$, 
such an index must exist: On the one hand, $\tilde{w}^*_1$ can be verified to 
be at least $\tilde{\lambda}>\tilde{\lambda}/2$ (by 
\lemref{lem:phase1} and the assumption $\gamma\geq 10^4(\lambda/\mu_2)^2$, 
hence $\gamma\geq 277\tilde{\lambda}^2$). On the other hand, if we let $t_1$ be 
the largest index $\leq \tilde{T}$ satisfying 
$\tilde{w}^*_{t_1}>\tilde{\lambda}/2$, we have by
\lemref{lem:woptstruct} that 
$\frac{\gamma}{\tilde{\lambda}} \geq \sum_{t=1}^{t_1}\tilde{w}^*_t \geq 
t_1\tilde{w}_{t_1}^*>  
\frac{t_1 \tilde{\lambda}}{2}$, which implies that $t_1 \leq 
\frac{2\gamma}{\tilde{\lambda}^2}$, which is less than $\tilde{T}/2$ by the 
assumption on $\tilde{T}$ being large enough. Therefore, $t_0$ is at most 
$\tilde{T}/2$ as well.}. 
Using the displayed inequality above, we get that for any $k\leq \tilde{T}-t_0$,
\[
\tilde{w}^*_{t_0+k} ~\geq~
\frac{(\tilde{w}^*_{t_0+k-1})^2}{9\tilde{\lambda}}~\geq~
\frac{(\tilde{w}^*_{t_0+k-2})^4}{(9\tilde{\lambda})^3}~\geq~\cdots~\geq~
9\tilde{\lambda}\left(\frac{\tilde{w}^*_{t_0}}{9\tilde{\lambda}}\right)^{2^k}
~>~
9\tilde{\lambda}\left(\frac{\tilde{\lambda}/2}{9\tilde{\lambda}}\right)^{2^k},
\]
so we get $\tilde{w}^*_{t_0+k} ~\geq~ 9\tilde{\lambda}\cdot (18)^{-2^k}$ as
required
\end{proof}

\begin{lemma}\label{lem:phase3}
	$\sum_{i=1}^{\tilde{T}}(\tilde{w}^*_i)^2 \leq 
	2\gamma^{7/4}/\tilde{\lambda}^{3/2}$
\end{lemma}
\begin{proof}
	We need to upper bound the squared Euclidean norm of 
	$(\tilde{w}^*_1,\ldots,\tilde{w}^*_{\tilde{T}})$. Note that for any vector 
	$\bw$, we have $\norm{\bw}^2=\sum_i w_i^2 \leq (\max_i |w_i|)\sum_i |w_i| = 
	\norm{\bw}_{\infty}\norm{\bw}_1$. Thus, by \lemref{lem:woptstruct}, 
	\eqref{eq:w1upbound}, and the assumption that $\gamma\geq 
	10^4(\lambda/\mu_2)^2>277\tilde{\lambda}^2$, the squared norm is at most
\[
\left(\sqrt{\gamma}+\sqrt{\frac{2\gamma^{3/2}}{\tilde{\lambda}}}\right)\cdot 
\frac{\gamma}{\tilde{\lambda}}
~=~\left(1+\sqrt{\frac{2\gamma^{1/2}}{\tilde{\lambda}}}\right)\cdot 
\frac{\gamma^{3/2}}{\tilde{\lambda}}
~\leq~\left(\sqrt{\frac{\gamma^{1/2}}{\sqrt{277}\tilde{\lambda}}}+\sqrt{\frac{2\gamma^{1/2}}{\tilde{\lambda}}}\right)\cdot
 \frac{\gamma^{3/2}}{\tilde{\lambda}},
\]
which is at most $2\sqrt{\gamma^{1/2}/\tilde{\lambda}}\cdot 
\gamma^{3/2}/\tilde{\lambda} = 2\gamma^{7/4}/\tilde{\lambda}^{3/2}$
\end{proof}

\begin{lemma}\label{lem:phase4}
$\bw^*=\arg\min_{\bw} f(\bw)$ satisfies $\inner{\bv_i,\bw^*}=\tilde{w}_i^*$ for
all $i=1,\ldots,\tilde{T}$, where $\tilde{\bw}^*=\arg\min_{\bw} \tilde{f}(\bw)$.
\end{lemma}

\begin{proof}
First, we argue that $\tilde{\bw}^*$, whic minimizes
\[
\tilde{f}(\bw) =
\frac{1}{3}\sum_{i=1}^{\tilde{T}-1}|w_i-w_{i+1}|^3+\frac{\tilde{\lambda}}{2}\norm{\bw}^2-\gamma\cdot
 w_1~,
\]
also minimizes
\[
\hat{f}(\bw) =
\sum_{i=1}^{\tilde{T}-1}g(w_i-w_{i+1})+\frac{\tilde{\lambda}}{2}\norm{\bw}^2-\gamma\cdot
 w_1~.
\]
To see this, note that $\tilde{f}$ and $\hat{f}$ differ
only in that $g(x)$ is replaced by $\frac{1}{3}|x|^3$. By definition of $g$,
we have that $g(x)$ and $\frac{1}{3}|x|^3$ coincide for any $|x|\leq \Delta$, from
which it is easily verified that $f$ and $\tilde{f}$ have the same values and
gradients at any $\bw$ for which $|w_i-w_{i+1}|\leq \Delta$ for all $i\leq \tilde{T}-1$. By
\lemref{lem:woptstruct} and the assumption $\Delta\geq \sqrt{\gamma}$, 
the global minimizer $\tilde{\bw}^*$ of $\tilde{f}$
belongs to this set, and therefore $\nabla \tilde{f}(\tilde{\bw}^*)=\nabla
\hat{f}(\tilde{\bw}^*)=\mathbf{0}$. But $\hat{f}$ is strongly convex, hence has
a unique point (the global minimizer) at which the gradient of $\hat{f}$ is 
zero,
hence $\tilde{\bw}^*$ is indeed the global minimizer of $\hat{f}$.

Next, since the global minimizer is invariant to multiplying the function by a
fixed positive factor, we get that
$\tilde{\bw}^*$ is also the global minimizer of
\begin{align*}
\frac{\mu_2}{12}\hat{f}(\bw) &=
\frac{\mu_2}{12}\left(\sum_{i=1}^{\tilde{T}-1}g(w_i-w_{i+1})+
\frac{\tilde{\lambda}}{2}\norm{\bw}^2-\gamma\cdot w_1\right)\\
&=
\frac{\mu_2}{12}\left(\sum_{i=1}^{\tilde{T}-1}g(w_i-w_{i+1})-\gamma\cdot 
w_1\right)+\frac{\lambda}{2}\norm{\bw}^2,
\end{align*}
where in the last step we used the fact that $\tilde{\lambda}=12\lambda/\mu_2$.
Recalling that
\[
f(\bw) =
\frac{\mu}{12}\left(\sum_{i=1}^{\tilde{T}-1}g(\inner{\bv_i,\bw}-\inner{\bv_{i+1},\bw})
-\inner{\bv_1,\bw}\right)+\frac{\lambda}{2}\norm{\bw}^2,
\]
and that $\bv_1,\bv_2,\ldots$ are orthogonal, we can write $f(\bw)$ as
$\frac{\mu}{12}\cdot
\hat{f}(V\bw)$, where $V$ is any orthogonal matrix with the first $\tilde{T}$ 
columns being $\bv_1,\ldots,\bv_{\tilde{T}}$.
Therefore, the
minimizer $\bw^*$ of $f$ satisfies $V\bw^* =(\inner{\bv_1,\bw^*},\inner{\bv_2,\bw_2^*},\ldots)= \tilde{\bw}^*$.
\end{proof}

\begin{lemma}\label{lem:phase5}
$\norm{\bw^*}^2 = \sum_{i=1}^{\tilde{T}}\inner{\bv_i,\bw^*}^2$
\end{lemma}

\begin{proof}
$f(\bw)$ is a function which can be written in the form 
$h(\inner{\bv_1,\bw},\inner{\bv_2,\bw},\ldots,\inner{\bv_{\tilde{T}},\bw})+\frac{\lambda}{2}\norm{\bw}^2$,
so by the Representer theorem, its minimizer $\bw^*$ must lie in the span of 
$\bv_1,\bv_2,\ldots,\bv_{\tilde{T}}$. Moreover, since these vectors are 
orthogonal and of unit 
norm, we have $\bw^* = \sum_{i=1}^{\tilde{T}}\inner{\bv_i,\bw^*}\bv_i$, and thus
\[
\norm{\bw^*}^2 = 
\left\langle\sum_{i=1}^{\tilde{T}}\inner{\bv_i,\bw^*}\bv_i,\sum_{j=1}^{\tilde{T}}
\inner{\bv_j,\bw^*}\bv_j\right\rangle = 
\sum_{i,j=1}^{\tilde{T}}\inner{\bv_i,\bw^*} 
\inner{\bv_j,\bw^*}\inner{\bv_i,\bv_j}=\sum_{i=1}^{\tilde{T}}\inner{\bv_i,\bw^*}^2.
\]
\end{proof}

\subsection{Oracle Complexity Lower Bound}\label{subsec:mainproof}

In this subsection, we prove the following oracle complexity lower bound, 
depending on the free parameter $\gamma$: 

\begin{proposition}\label{prop:lowbound}
Assume that $\epsilon < 
\min\left\{\frac{108^2\cdot\lambda^3}{\mu_2^2},\frac{\gamma\lambda}{8}\right\}$. 
Under the conditions of \propref{prop:opt}, it is possible to choose the 
vectors $\bv_1,\bv_2,\ldots,\bv_{\tilde{T}}$ in the function $f$, such that the 
number of iterations $T$ required to have
$f(\bw_T)-f(\bw^*)\leq \epsilon$ is at least
\[
\max\left\{\frac{\gamma^{1/4}}{7\sqrt{12\lambda/\mu_2}}~,~\log_2  
\log_{18}\left(\frac{108^2\cdot\lambda^3 
}{\mu_2^2\epsilon}\right) -1\right\}
\]
\end{proposition}

To prove the theorem, we will need the following key lemma, which establishes 
that oracle information at certain points $\bw$ do not leak any information on 
some of the $\bv_1,\bv_2,\ldots$ vectors. 

\begin{lemma}\label{lem:info}
	For any $\bw\in \reals^d$ orthogonal to 
	$\bv_t,\bv_{t+1},\ldots,\bv_{\tilde{T}}$ , it holds that $f(\bw), \nabla 
	f(\bw),\nabla^2 f(\bw)$ do not depend on 
	$\bv_{t+1},\bv_{t+2},\ldots,\bv_{\tilde{T}}$.
\end{lemma}

\begin{proof}
Since the regularization term $\frac{\lambda}{2}\norm{\bw}^2$ doesn't depend on 
$\bv_{t+1},\bv_{t+2},\ldots,\bv_{\tilde{T}}$ we can define $h(\bw) \triangleq 
f(\bw)-\frac{\lambda}{2}\norm{\bw}^2$
and prove the result on $h(\bw)$.
Using the definition of $h$ and differentiating, we have that

\[
 h(\bw) =\frac{\mu_2}{12} 
 \left(\sum_{i=1}^{\tilde{T}-1}g(\inner{\bv_i-\bv_{i+1},\bw}) - 
 \gamma\inner{\bv_1,\bw}\right)
 \]
 \[
 \nabla h(\bw) =\frac{\mu_2}{12} 
 \left(\sum_{i=1}^{\tilde{T}-1}g'(\inner{\bv_i-\bv_{i+1},\bw})(\bv_{i}-\bv_{i+1})
 - \gamma\bv_1 \right)
 \]
 \[
 \nabla^2 h(\bw) =\frac{\mu_2}{12} \left( 
 \sum_{i=1}^{\tilde{T}-1}g''(\inner{\bv_i-\bv_{i+1},\bw})(\bv_{i}-\bv_{i+1})(\bv_{i}-\bv_{i+1})^T
  \right)
 \]
 By the assumption $\inner{\bv_{t},\bw}=\inner{\bv_{t+1},\bw}=\ldots=0$,  and 
 the fact that $g(0)=g'(0)=g''(0)=0$, we have that 
 $g(\inner{\bv_i-\bv_{i+1},\bw})=g'(\inner{\bv_i-\bv_{i+1},\bw})=g''(\inner{\bv_i-\bv_{i+1},\bw})=0$
  for all $i\in \{t,t+1,\ldots,\tilde{T}-1\}$. Therefore, it is easily verified 
 that the expressions above indeed do not depend on 
 $\bv_{t+1},\bv_{t+2},\ldots,\bv_{\tilde{T}}$.
\end{proof}

Let us now fix any number of iterations $T\leq \tilde{T}$. Using the previous 
results, we can 
provide a way to pick $\bv_1,\bv_2,\ldots,\bv_{\tilde{T}}$ 
for any deterministic algorithm, such that we can provide a lower bound for the 
number of second-order oracle calls.

\begin{itemize}
	\item First, we compute $\bw_1$ (which is possible since the algorithm is 
	deterministic and $\bw_1$ is chosen before any oracle calls are made).
	\item We pick $\bv_1$ to be some unit vector orthogonal to $\bw_1$. 
	Assuming $\bv_2,\ldots,\bv_{\tilde{T}}$ will also be orthogonal to $\bw_1$ 
	(which 
	will be ensured by the construction which follows), we have by 
	\lemref{lem:info} that the information $F(\bw_1),\nabla F(\bw_1),\nabla^2 
	F(\bw_1)$ provided by the oracle to the algorithm does not depend on 
	$\{\bv_2,\ldots,\bv_{\tilde{T}}\}$, and thus depends only on $\bv_1$ which 
	was 
	already fixed. Since the algorithm is deterministic, this fixes the next 
	query point $\bw_2$.
	\item For $t=2,3,\ldots,T-1$, we repeat the process above: We compute 
	$\bw_t$, and pick $\bv_{t}$ to be some unit vectors orthogonal to 
	$\bw_1,\bw_2,\ldots,\bw_t$, as well as all previously constructed $\bv$'s 
	(this is always possible since the dimension is sufficiently large). By 
	\lemref{lem:info}, as long as all vectors thus constructed are orthogonal 
	to $\bw_t$, the information $\{F(\bw_t),\nabla F(\bw_t),\nabla^2 
	F(\bw_t)\}$ provided to the algorithm does not depend on 
	$\bv_{t+1},\ldots,\bv_{\tilde{T}}$, and only depends on 
	$\bv_1,\ldots,\bv_t$ which 
	were already determined. Therefore, the next query point $\bw_{t+1}$ is 
	fixed. 
	\item At the end of the process, we pick  
	$\bv_{T},\bv_{T+1},\ldots,\bv_{\tilde{T}}$ to be some unit vectors 
	orthogonal to all previously chosen $\bv$'s as well as 
	$\bw_1,\ldots,\bw_T$ (this is possible since the dimension is large 
	enough). 
\end{itemize}

Using the facts that $\bv_1,\bv_2,\ldots,\bv_{\tilde{T}}$ are orthogonal (and 
thus 
act as an orthonormal basis to a subspace of $\reals^d$), that $\bw_T$ is 
orthogonal to $\bv_T,\bv_{T+1},\ldots\bv_{\tilde{T}}$, and that $t_0+T\leq 
\frac{\tilde{T}}{2}+\frac{\tilde{T}}{2}=\tilde{T}$ (where $t_0$ is as defined 
in \propref{prop:opt}), we have

\[
\norm{\bw_T-\bw^*}^2 \geq \sum_{i=1}^{\tilde{T}} \inner{\bv_i,\bw_T-\bw^*}^2 
\geq 
\inner{\bv_{t_0+T},\bw_T-\bw^*}^2 = 
\inner{\bv_{t_0+T},\bw^*}^2~.
\]
By \propref{prop:opt}, we can 
lower bound the above by
\[
\frac{108^2\lambda^2}{\mu_2^2}\cdot(18)^{-2^{(T+1)}}.
\]
Using the strong convexity of $f$, we therefore get
\[
f(\bw_T)-f(\bw^*) \geq \frac{\lambda}{2}\cdot 
\norm{\bw_T-\bw^*}^2 \geq 
\frac{108^2\cdot\lambda^3}{\mu_2^2} (18)^{-2^{(T+1)}}.
\]
To make the right-hand side smaller than $\epsilon$, $T$ must satisfy
\[
(18)^{-2^{(T+1)}} \leq \frac{\mu_2^2\epsilon}{108^2\cdot\lambda^3 }
\]
Which is equivalent to
\[
2^{(T+1)} \geq \log_{18}\left(\frac{108^2\cdot\lambda^3 
}{\mu_2^2\epsilon}\right)~.
\]
Assuming $\epsilon < 
\frac{108^2\cdot\lambda^3}{\mu_2^2}$, then

\[
T ~\geq~ \log_2  \log_{18}\left(\frac{108^2\cdot\lambda^3 
}{\mu_2^2\epsilon}\right) -1~.
\]

We now turn to argue that we can also lower bound $T$ by 
$\frac{\gamma^{1/4}}{7\sqrt{12\lambda/\mu_2}}$. Otherwise, suppose by 
contradiction that we can have $f(\bw_T)-f(\bw^*)\leq \epsilon$ for some
$T< \frac{\gamma^{1/4}}{7\sqrt{12\lambda/\mu_2}}$. From \propref{prop:opt} we 
know 
that 
\[
\inner{\bv_T,\bw^*} ~\geq~
\frac{\gamma^{3/4}}{7\sqrt{12\lambda/\mu_2}}
	+\sqrt{\gamma}\left(\frac{1}{2}-T\right)~,
\]
so as before, we have that
\[
\norm{\bw_T-\bw^*}^2 \geq \sum_{i=1}^{\tilde{T}} \inner{\bv_i,\bw_T-\bw^*}^2 
\geq 
\inner{\bv_{T},\bw_T-\bw^*}^2 = 
\inner{\bv_{T},\bw^*}^2,
\]
and thus
\[
f(\bw_T)-f(\bw^*) ~\geq~ 
\frac{\lambda}{2}\cdot 
\norm{\bw_T-\bw^*}^2 ~\geq~ \frac{\lambda}{2}\cdot 
\inner{\bv_T,\bw^*}^2~\geq~\frac{\lambda}{2}
\left(\frac{\gamma^{3/4}}{7\sqrt{12\lambda/\mu_2}}
	+\sqrt{\gamma}\left(\frac{1}{2}-T\right)\right)^2.
\]

To make the right-hand side smaller than $\epsilon$, $T$ must satisfy
\[
\left(\frac{\gamma^{3/4}}{7\sqrt{12\lambda/\mu_2}}
	+\sqrt{\gamma}\left(\frac{1}{2}-T\right)\right)^2
\leq \frac{2\epsilon}{\lambda},
\]
or equivalently
\[
T ~\geq~ 
\frac{\gamma^{1/4}}{7\sqrt{12\lambda/\mu_2}}+\frac{1}{2}
- \sqrt{\frac{2\epsilon}{\gamma\lambda}}.
\]
But since we assume $\epsilon < \frac{\gamma\lambda}{8}$, this is at least 
$\frac{\gamma^{1/4}}{7\sqrt{12\lambda/\mu_2}}$, contradicting our earlier 
assumption. 

Overall, we showed that $T$ is lower bounded by both 
$\frac{\gamma^{1/4}}{7\sqrt{12\lambda/\mu_2}}$, as well as $\log_2 
\log_{18}\left(\frac{108^2\cdot\lambda^3}{\mu_2^2 \epsilon}\right) -1$, 
hence
proving \propref{prop:lowbound}.

\subsection{Setting the $\gamma,\Delta$ Parameters}\label{subsec:smoothstrong}

In the following lemma, we establish the strong convexity and smoothness 
parameters of $f$ (depending on the parameter $\Delta$ which is still 
free at this point).

\begin{lemma}\label{lem:smoothstrong}
	$f$ is $\lambda$-strongly convex and twice-differentiable, with
	$\mu_2$-Lipschitz Hessians and 
	$\left(\frac{2\mu_2\Delta}{3}+\lambda\right)$-Lipschitz gradients.
\end{lemma}

\begin{proof}
	Since $f$ is a sum of convex, twice-differentiable functions and the
	$\lambda$-strongly convex
	function $\frac{\lambda}{2}\norm{\bw}^2$, it is clearly $\lambda$-strongly
	convex and twice-differentiable. Thus, it only remains to calculate the 
	Lipschitz parameter of the gradients and Hessians. 
	
	To simplify the proof, we note that Lipschitz smoothness is a property
	invariant to the coordinate system used, so we can assume without loss of
	generality that $\bv_1,\bv_2,\ldots,\bv_{\tilde{T}}$ correspond to the 
	standard basis
	$\be_1,\be_2,\ldots,\be_{\tilde{T}}$, and consider the Lipschitz properties 
	of the function
	\[
	\hat{f}(\bw) =
	\frac{\mu_2}{12}\left(\sum_{i=1}^{\tilde{T}-1}g(w_i-w_{i+1})-\gamma\cdot 
	w_1\right)+\frac{\lambda}{2}\norm{\bw}^2
	\]
	
	By definition of $g$, it is easily verified that 
	\[
	g''(x) = 2\cdot 
	\min\{\Delta,|x|\}~,
	\]
	which is a $2$-Lipschitz function bounded in $[0,2\Delta]$. This implies 
	that $g'(x)$ is $2\Delta$-Lipschitz.
	Letting $\br_i:=\be_i-\be_{i+1}$, we can write $\hat{f}$ as
	\[
	\hat{f}(\bw) = 
	\frac{\mu_2}{12}\left(\sum_{i=1}^{\tilde{T}-1}g(\inner{\br_i,\bw})-\gamma\cdot
	 w_1\right)+\frac{\lambda}{2}\norm{\bw}^2.
	\]
	Differentiating twice, we get
	\[
	\nabla^2 \hat{f}(\bw) = \frac{\mu_2}{12}
	\sum_{i=1}^{\tilde{T}-1}g''(\inner{\br_i,\bw})\cdot\br_i
	\br_i^\top+\lambda I.
	\]
	Since this is a sum of positive-semidefinite matrices with non-negative
	coefficients (as we showed that $g''(x)\in [0,2\Delta]$ for all $x$), it 
	follows 
	that
	its spectral norm is at most
	\[
	\frac{\mu_2\Delta}{6}\cdot\left\|\sum_{i=1}^{\tilde{T}-1}\br_i\br_i^\top\right\|+\lambda,
	\]
	and the first term equals 
	\begin{align*}
		\frac{\mu_2\Delta}{6}\cdot 	
		\max_{\bx}\frac{\sum_{i=1}^{\tilde{T}-1}\inner{\br_i,\bx}^2}{\norm{\bx}^2}
		&~=~
		\frac{\mu_2\Delta}{6}\cdot 	
		\max_{\bx}\frac{\sum_{i=1}^{\tilde{T}-1}(x_i-x_{i+1})^2}{\norm{\bx}^2}\\
		&~\leq~
		\frac{\mu_2\Delta}{6}\cdot
		\max_{\bx}\frac{\sum_{i=1}^{\tilde{T}-1}(2x_i^2+2x_{i+1}^2)}{\sum_{i=1}^{d}x_i^2}~\leq~
		\frac{2\mu_2\Delta}{3} .
	\end{align*}
	Overall, we showed that $\norm{\nabla^2 \hat{f}(\bw)}\leq
	\frac{2\mu_2\Delta}{3}+\lambda$,
	so the gradients of $f$ are 
	$\left(\frac{2\mu_2\Delta}{3}+\lambda\right)$-Lipschitz.
	
	It remains to show that $\nabla^2 \hat{f}(\bw)$ is $\mu_2$-Lipschitz. Using 
	the
	formula for $\nabla^2 \hat{f}(\bw)$, and recalling that $g''(x)$ is
	$2$-Lipschitz, and $\norm{\br_i}=\sqrt{2}$ by definition, we have that for
	any $\bw,\tilde{\bw}$,
	\begin{align*}
		\norm{\nabla^2 \hat{f}(\bw)-\nabla^2 \hat{f}(\tilde{\bw})}
		&~=~\frac{\mu_2}{12}\cdot\left\|{\sum_{i=1}^{\tilde{T}-1}(g''(\inner{\br_i,\bw})-g''(\inner{\br_i,\tilde{\bw}}))\cdot
			\br_i\br_i^\top}\right\|\\
		&~\leq~\frac{\mu_2}{12}\cdot\left\|{\sum_{i=1}^{\tilde{T}-1}|g''(\inner{\br_i,\bw})-g''(\inner{\br_i,\tilde{\bw}})|\cdot
			\br_i\br_i^\top}\right\|\\
		&~\leq~
		\frac{\mu_2}{12}\cdot\left\|{\sum_{i=1}^{\tilde{T}-1}2|\inner{\br_i,\bw-\bw'}|\cdot
			\br_i\br_i^\top}\right\|\\
		&~\leq~
		\frac{\mu_2}{12}\cdot2\sqrt{2}\cdot\norm{\bw-\tilde{\bw}}\cdot\left\|
		\sum_{i=1}^{\tilde{T}-1}\br_i\br_i^\top\right\|.
	\end{align*}
	Using the same calculations as earlier, we have
	$\left\|\sum_{i=1}^{\tilde{T}-1}\br_i\br_i^\top\right\|\leq 4$, and 
	therefore we
	showed overall that
	\[
	\norm{\nabla^2 \hat{f}(\bw)-\nabla^2 \hat{f}(\tilde{\bw})} ~\leq~
	\frac{\mu_2\cdot 8\sqrt{2}}{12}\cdot \norm{\bw-\tilde{\bw}} < \mu_2\cdot
	\norm{\bw-\tilde{\bw}},
	\]
	hence $\nabla^2 \hat{f}(\bw)$ is $\mu_2$-Lipschitz.
\end{proof}

We now collect the ingredients necessary to fix $\gamma,\Delta$ and hence prove 
our theorem. Combining the 
previous lemma, \propref{prop:opt} and \propref{prop:lowbound}, and recalling 
that 
we want $f$ to have $\mu_1$-Lipschitz gradients and $\mu_2$-Lipschitz Hessians, 
with an optimizer $\bw^*$ satisfying $\norm{\bw^*}\leq D$, we have an oracle 
complexity lower bound of the form
\begin{equation}\label{eq:finalbound}
T~\geq~ \max\left\{\frac{\gamma^{1/4}}{7\sqrt{12\lambda/\mu_2}}~,~\log_2
\log_{18}\left(\frac{108^2\cdot\lambda^3 
}{\mu_2^2\epsilon}\right) -1\right\},
\end{equation}
assuming the following conditions:
\[
\gamma\geq 10^4\left(\frac{\lambda}{\mu_2}\right)^2~~,~~\Delta\geq 
\sqrt{\gamma}~~,~~
\epsilon < 
\min\left\{\frac{108^2\cdot\lambda^3}{\mu_2^2},\frac{\gamma\lambda}{8}\right\}~~,~~
\frac{2\mu_2\Delta}{3}+\lambda\leq \mu_1~~,~~
\sqrt{\frac{2\gamma^{7/4}}{(12\lambda/\mu_2)^{3/2}}}\leq D~.
\]
Picking $\Delta=\sqrt{\gamma}$, using the fact that $\mu_1\geq \lambda$ (as any 
$\lambda$-strongly convex function must have gradients with Lipschitz parameter 
at least $\lambda$), and 
rewriting the last two conditions, this is equivalent 
to
\[
\gamma\geq 10^4\left(\frac{\lambda}{\mu_2}\right)^2~~,~~
\epsilon < 
\min\left\{\frac{108^2\cdot\lambda^3}{\mu_2^2},\frac{\gamma\lambda}{8}\right\}~~,~~
\gamma\leq \left(\frac{3(\mu_1-\lambda)}{2\mu_2}\right)^2~~,~~
\gamma\leq \sqrt[7]{\frac{D^8(12\lambda)^6}{2^4\mu_2^6}}~.
\]
Since the first condition needs to hold anyway, we can allow ourself to make 
the second condition stronger, by substituting $10^4(\lambda/\mu_2)^2$ in lieu 
of $\gamma$ in the second condition. Doing this, simplifying, and merging the 
last two conditions, the set of condition above is implied by requiring

 \[
\gamma\geq 10^4\left(\frac{\lambda}{\mu_2}\right)^2~~,~~
\epsilon < \frac{10^4 \lambda^3}{8\mu_2^2}
~~,~~
\gamma\leq \min\left\{\left(\frac{3(\mu_1-\lambda)}{2\mu_2}\right)^2~,~
\sqrt[7]{\frac{D^8(12\lambda)^6}{2^4\mu_2^6}}\right\}~.
\]
Clearly, to make the lower bound in \eqref{eq:finalbound} as large as possible, 
we should
pick the largest possible $\gamma$, namely $
\gamma=\min\left\{\left(\frac{3(\mu_1-\lambda)}{2\mu_2}\right)^2,
\sqrt[7]{\frac{D^8(12\lambda)^6}{2^4\mu_2^6}}\right\}$, 
and to ensure that the other conditions hold, require that
\[
\min\left\{\left(\frac{3(\mu_1-\lambda)}{2\mu_2}\right)^2,
\sqrt[7]{\frac{D^8(12\lambda)^6}{2^4\mu_2^6}}\right\}
~\geq~ 10^4\left(\frac{\lambda}{\mu_2}\right)^2~~,~~
\epsilon < \frac{10^4 \lambda^3}{8\mu_2^2}.
\]
Simplifying a bit, these two conditions are implied by requiring
\begin{equation}\label{eq:finalconditions}
\frac{\mu_1}{\lambda}\geq 68~~,~~\frac{\mu_2}{\lambda}D\geq 694~~,~~ \epsilon < \frac{10^4 
\lambda^3}{8\mu_2^2},
\end{equation}

Finally, let us plug our choice of 
$\gamma=\min\left\{\left(\frac{3(\mu_1-\lambda)}{2\mu_2}\right)^2,
\sqrt[7]{\frac{D^8(12\lambda)^6}{2^4\mu_2^6}}\right\}$ into the lower bound in 
\eqref{eq:finalbound}. We thus get an oracle complexity lower bound of
\begin{align*}
&\max\left\{\min\left\{\frac{\sqrt{\mu_2}}{7\sqrt{12\lambda}}\cdot\sqrt{\frac{3(\mu_1-\lambda)}
{2\mu_2}}~,~\frac{\sqrt{\mu_2}}{7\sqrt{12\lambda}}\cdot 
\frac{D^{2/7}(12\lambda)^{3/14}}{2^{1/7}\mu_2^{3/14}}\right\}~,~
\log_2
\log_{18}\left(\frac{108^2\cdot\lambda^3 
}{\mu_2^2\epsilon}\right) -1
\right\}\\
&=
\max\left\{\min\left\{\frac{1}{14}\sqrt{\frac{\mu_1-\lambda}{2\lambda}}~,~
\frac{(D\mu_2/12\lambda)^{2/7}}{7\cdot 2^{1/7}}\right\}~,~
\log_2
\log_{18}\left(\frac{108^2\cdot\lambda^3 
}{\mu_2^2\epsilon}\right) -1
\right\},
\end{align*}
under the conditions of \eqref{eq:finalconditions}. 

To simplify the bound a bit, we note that we can lower bound 
$\mu_1-\lambda$ by $\frac{67}{68}\mu_1$ (possible by 
\eqref{eq:finalconditions}), 
and lower bound $\log_2
\log_{18}\left(\frac{108^2\cdot\lambda^3 }{\mu_2^2\epsilon}\right) -1$ by
$\frac{1}{2}\log\log_{18}\left(\frac{\lambda^3}{\mu_2^2\epsilon}\right)$, by 
assuming that 
$\epsilon\leq c\lambda^3/\mu_2^2$ for some small enough $c$ 
(in other words, increasing the constant
in the third condition in \eqref{eq:finalconditions}). Finally, using the fact 
that $\max\{a,b\}\geq (a+b)/2$, the result in the theorem follows.

\section{Proof of \thmref{thm:main2}}

Similarly to the strongly convex case, we will assume without loss of 
generality that the algorithm initializes at 
$\bw_1=\mathbf{0}$, since otherwise one can simply replace the
``hard'' function $f(\bw)$ below by $f(\bw-\bw_1)$, and the same proof holds 
verbatim. Thus, the theorem requires that our function has a minimizer $\bw^*$ 
satisfying $\norm{\bw^*}\leq D$.

Define $g:\reals\mapsto\reals$ as
\[
g(x) = \begin{cases} \frac{1}{3}|x|^3 & |x|\leq \Delta\\
\Delta x^2-\Delta^2|x|+\frac{1}{3}\Delta^3&|x|>\Delta ,
\end{cases}.
\]
where $\Delta \triangleq \frac{3 \mu_1}{2 \mu_2}$.
$g$ can be easily verified to be twice continuously differentiable. Assume that 
$d\geq 2T$, and let
$\bv_1,\bv_2,\ldots,\bv_{T}$ be orthogonal unit vectors in $\reals^d$ which 
will be
specified
later. Given $T$, and letting $\gamma>0$ be a parameter to be specified later, 
define the function $f_T$ as
\[
f_T(\bw) =
\frac{\mu_2}{12}\left(g(\inner{\bv_1,\bw})+g(\inner{\bv_T,\bw})+\sum_{i=1}^{T-1}g(\inner{\bv_i,\bw}-\inner{\bv_{i+1},\bw})-\gamma\inner{\bv_1,\bw}\right).
\]
This function is easily shown to be convex and twice-differentiable, with
$\mu_1$-Lipschitz gradients and $\mu_2$-Lipschitz Hessians (the proof is 
identical to the proof of \lemref{lem:smoothstrong}). Our goal will be to show 
a lower bound on the optimization error using this 
type of function.

\subsection{Minimizer of $f_T$}\label{subsec:min_f_T_convex}

In this subsection, we analyze the properties of a minimizer of $f_T$. To that 
end, we introduce the following function in $\reals^T$:
\[
\hat{f}_T(\bw)=g(w_1)+g(w_T) + 
\sum_{i=1}^{T-1}g(w_i-w_{i+1})-\gamma w_1.
\]
It is easily verified that the minimal values of $\frac{\mu_2}{12}\hat{f}_T$ and $f_T$ are the 
same, and moreover, if $\hat{\bw}\in\reals^T$ is a minimizer of $\hat{f}_T$, 
then $\bw^*=\sum_{j=1}^{T}\hat{w}^*_j\cdot \bv_j\in \reals^d$ is a minimizer of 
$f_T$, and with the same Euclidean norm as $\hat{\bw}^*$.

We begin with the following technical lemma:
\begin{lemma}\label{lem:gen_sol}
	$\hat{f}_T$ has a unique minimizer $\hat{\bw}^* \in \reals^T$, which  
	satisfies
	\[
	\hat{w}^*_t = \delta\cdot(T+1)\cdot\left(1-\frac{t}{T+1}\right)~,
	\]
	for all $t=1,2,\ldots,T$, where $\delta$ is non-negative and independent of 
	t. Moreover,
	\[
	g'(\hat{w}^*_1) + g'(\hat{w}^*_T) = \gamma ~.
	\]	
\end{lemma}
\begin{proof}
Taking the derivative and setting to zero, we get that the 
\[
g'(\hat{w}^*_1)+g'(\hat{w}^*_1-\hat{w}^*_2)=\gamma~,~~g'(\hat{w}^*_{T-1}-\hat{w}^*_T)=g'(\hat{w}^*_T)
\]
as well as
\[
g'(\hat{w}^*_{j-1}-\hat{w}^*_j)=g'(\hat{w}^*_{j}-\hat{w}^*_{j+1})
\]
for all $j\in \{2,3,\ldots,T-1\}$. By definition of $g$, it is easily verified 
that $g'$ is 
a strictly monotonic (hence invertible) function, so the above implies 
$\hat{w}^*_{j-1}-\hat{w}^*_j=\hat{w}^*_{j}-\hat{w}^*_{j+1}$ for all $j\in 
\{2,3,\ldots,T-1\}$, as well as $\hat{w}^*_{T-1}-\hat{w}^*_T=\hat{w}^*_T$. From 
this, it follows by straightforward induction that $\hat{w}^*_{T+1-t}=t\cdot 
\hat{w}^*_T$, from which the first displayed equation in the lemma follows. 
This also implies $g'(T\hat{w}^*_T)+g'(\hat{w}^*_T)=\gamma$, and 
since $g'$ is strictly monotonic, we have that $\hat{w}^*_T$ is uniquely 
defined, and since the other coordinates of $\hat{\bw}^*$ are also uniquely 
defined given $\hat{w}^*_T$, we get that $\hat{\bw}^*$ is unique. Finally, 
$\delta$ (and hence $\hat{w}_t^*$ for all $t$) is necessarily non-negative, 
since otherwise $\hat{w}^*_1$ is negative, which would imply 
$\hat{f}_T(\hat{\bw}^*)>0$, even though $\hat{f}_T(\mathbf{0})=0$, violating 
the fact that $\hat{\bw}^*$ minimizes $\hat{f}_T$. 
\end{proof}

The main technical result in this subsection is the following proposition, 
which characterizes $\norm{\hat{\bw}^*}$ and $\hat{f}_T(\hat{\bw}^*)$ under 
various parameter regimes. By the discussion above and definition of $f_T$, we 
have 
\begin{equation}\label{eq:fhatf}
\norm{\bw^*}=\norm{\hat{\bw}^*}~~~\text{and}~~~f_T(\bw^*)=\frac{\mu_2}{12}\cdot\hat{f}_T(\hat{\bw}^*)~,
\end{equation}
which will be used in the remainder of the proof of our theorem.
\begin{proposition}\label{prop:opt_convex}
The function $\hat{f}_T$ and its minimizer $\hat{\bw}^*$ has the following 
properties, depending on the values of $\gamma,\Delta,T$:
\begin{center}
	\begin{tabular}{|c|c|c|c|}
		\cline{2-4}
		\multicolumn{1}{c|}{} & $\gamma \leq \frac{\Delta^2 \left( 
		1+T^2\right)}{T^2}$ & $\frac{\Delta^2 \left( 1+T^2\right)}{T^2} <\gamma 
		\leq 2 \Delta^2T$ & $\gamma > 2\Delta^2T$ \\
		\hline
		$\hat{f}_T(\hat{\bw}^*)$  & 
		$-\frac{2\mu_2\gamma^{3/2}T}{3\sqrt{\left(1+T^2\right)}}$ &
		\shortstack{$\frac{1}{3} T\delta^{3}+ \Delta T^2 \delta^{2} -T\left( 
		\Delta^2+\gamma\right)\delta + \frac{\Delta^3}{3}$ \\  $\delta = 
		-\Delta T + \Delta T\sqrt{1+\frac{\gamma + \Delta^2 }{\Delta^2 T^2}}$}  
		& 
		$-\frac{T\left(\gamma + 2\Delta^2\right)^2}{4\Delta(T+1)} 
		+\frac{(T+1)\Delta^3}{3}$ \\
		\hline
		$\norm{\hat{\bw}^*}^2$ & $\leq 
		\frac{\gamma(1+T)^3}{3\left(1+T^2\right)}$ & $\leq \frac{\left(\gamma + 
		\Delta^2\right)^2(T+1)^3}{12\Delta^2 T^2}$ & $\leq 
		\frac{(T+1)\left(\gamma + 2\Delta^2\right) ^2}{12\Delta^2}$ \\
		\hline
	\end{tabular}
\end{center}
	
\end{proposition}
\begin{proof}
	To prove the proposition, we will consider three regimes, depending on 
	$T,\delta,\Delta$: Namely, $T\delta\leq \Delta$, 
	$\frac{\Delta}{T}<\delta\leq\Delta$ and $\delta>\Delta$. We will show that 
	each regime corresponds to one of the three regimes specified in the 
	proposition, and prove the relevant bounds.
	
	\underline{\textbf{Case 1: $T\delta \leq \Delta$}}. In that case, 
	$\hat{w}^*_1,\hat{w}^*_T$ as well as $\hat{w}^*_i-\hat{w}^*_{i+1}$ for all 
	$i=2,\ldots,T-1$ in the definition of $\hat{f}_T$ all lie in the interval 
	where $g$ is a cubic function. Using \lemref{lem:gen_sol},
	\[
	g'(w^*_1) + g'(w^*_T) = w^{*2}_1 + w^{*2}_T = \gamma 
	\]
	hence
	\[
	\delta^2T^2+\delta^2 = \gamma
	\]
	and 
	\[
	\delta = \sqrt{\frac{\gamma}{1+T^2}}.
	\]
	Therefore, our condition $T\delta\leq\Delta$ is exactly equivalent to 
	$\gamma \leq \frac{\Delta^2 \left( 1+T^2\right)}{T^2}$, namely the first 
	regime discussed in the proposition. We now establish the relevant bounds:
	\begin{align*}
	\hat{f}_T(\hat{\bw}^*) = & 
	\frac{1}{3}\left(\frac{\sqrt{\gamma}(1+T)}{\sqrt{\left(1+T^2\right)}}\right)^3
	 \left[\left(1-\frac{1}{1+T} \right)^3 + \left(1-\frac{T}{1+T} \right)^3 
	\right]\\
	& ~~~~~~~~~+ \frac{1}{3}(T-1) 
	\left(\sqrt{\frac{\gamma}{\left(1+T^2\right)}}\right)^3 - 
	\frac{\gamma^{3/2}T}{\sqrt{\left(1+T^2\right)}} \\
	& =  \frac{1}{3} \left(\sqrt{\frac{\gamma}{\left(1+T^2\right)}}\right)^3 
	\left(T^3+1+T-1 \right) - \frac{\gamma^{3/2}T}{\sqrt{\left(1+T^2\right)}} \\
	& = \frac{\gamma^{3/2}T}{\sqrt{\left(1+T^2\right)}} \left(\frac{1}{3} - 1 
	\right)
	\end{align*}
	and	
	\begin{align*}
	\norm{\hat{\bw}^*}^2_2 = & \sum_{t=1}^T\hat{w}^{*2}_t 
	=\frac{\gamma(1+T)^2}{\left(1+T^2\right)}\sum_{t=1}^T \left(1-\frac{t}{1+T} 
	\right)^2 \\
	& = \frac{\gamma(1+T)^2}{\left(1+T^2\right)}\left( T - 
	\frac{2}{T+1}\sum_{t=1}^T t + \frac{1}{(T+1)^2}\sum_{t=1}^T t^2 \right) \\
	& \leq  \frac{\gamma(1+T)^2}{\left(1+T^2\right)}\left( T - \frac{2}{T+1} 
	\cdot \frac{T(T+1)}{2} + \frac{1}{(T+1)^2} \cdot \frac{(T+1)^3}{3} \right)\\
	& \leq \frac{\gamma(1+T)^3}{3\left(1+T^2\right)}~,
	\end{align*}
	where in the calculation above we used fact 
	$\sum_{t=1}^T t^2 \leq \int_1^{T+1}t^2dt < \frac{(T+1)^3}{3}$.
	
	\underline{\textbf{Case 2: $\frac{\Delta}{T}<\delta \leq \Delta$}}. In this 
	case, by \lemref{lem:gen_sol}, $\hat{w}^*_T \leq \Delta$ but $ \hat{w}^*_1 
	> \Delta$. Therefore, in the definition $\hat{f}_T(\hat{\bw}^*)$, 
	$g(\hat{w}^*_1)$ lies in the quadratic region of $g$, whereas 
	$g(\hat{w}^*_T)$ and $g'(\hat{w}^*_{i}-\hat{w}^*_{i+1})$ for all $i$ lies 
	in the cubic region of $g$. As a result,
	\[
	g'(w^*_1) + g'(w^*_T) = 2 \Delta w^*_1 -\Delta^2 +  w^{*2}_T  = \gamma.
	\]	
	Plugging in $w^*_T=\delta$ and $w^*_1=T\cdot\delta$, we get
	\[
	\delta^2 + 2\Delta \delta\cdot T -\left(\gamma + \Delta^2 \right)= 0~,
	\]
	and therefore (using the fact $\delta\geq 0$, see \lemref{lem:gen_sol}),
	\[
	\delta = -\Delta T + \Delta T\sqrt{1+\frac{\gamma + \Delta^2 }{\Delta^2 
	T^2}} ~.
	\]
	This, plus the assumption $\frac{\Delta}{T}<\delta \leq \Delta$, is 
	equivalent to $\frac{\Delta^2 \left( 1+T^2\right)}{T^2} <\gamma \leq 2 
	\Delta^2T$, hence showing that we are indeed in the second regime as 
	specified in our proposition. Turning to calculate the relevant bounds, we 
	have
	\[
	\hat{f}_T(\hat{\bw}^*) = T\delta^3+ \Delta^2T^2 \delta^2 -T\left( 
	\Delta^2+\gamma\right)\delta + \frac{\Delta^3}{3}~.
	\]
	Moreover,
	\[
	\norm{\hat{\bw}^*}^2 ~=~ 
	\delta^2(T+1)^2\sum_{t=1}^T \left(1-\frac{t}{1+T} \right)^2,
	\]
	which by definition of $\delta$ above and the inequality $\sqrt{1+x}\leq 
	1+\frac{1}{2}x$ for all $x\geq 0$, is at most $\frac{\left(\gamma + 
	\Delta^2\right)^2(T+1)^3}{12\Delta^2 T^2}$.
	
	\underline{\textbf{Case 3: $\delta > \Delta$}}. In this 
	case, by \lemref{lem:gen_sol}, we have $ \hat{w}^*_1>\hat{w}^*_T = 
	\hat{w}^*_i-\hat{w}^*_{i+1} > \Delta$, which implies that in the definition 
	of $\hat{f}_t(\hat{\bw}^*)$, these terms all lie in the quadratic region of 
	$g$. Therefore, 
	\[
	g'(w^*_1) + g'(w^*_T) = 2 \Delta w^*_1 -\Delta^2 + 2 \Delta w^*_T -\Delta^2 
	= \gamma ~,
	\]
	and thus
	\[
	2 \Delta (T+1)\delta = \gamma + 2 \Delta^2~, 
	\]
	or equivalently
	\[
	\delta = \frac{\gamma + 2\Delta^2}{2\Delta(T+1)}~.
	\]
	Note that this, plus our assumption $\delta>\Delta$, is equivalent to  
	$\gamma  > 2\Delta^2T $, which shows that we are indeed in the third regime 
	as specified in our proposition. Turning to calculate $\norm{\hat{\bw}^*}$ 
	and $\hat{f}_T(\hat{\bw}^*)$, we have
	\begin{align*}
	\hat{f}_T(\hat{\bw}^*) &= \Delta \hat{w}^{*2}_1 - \Delta^2 \hat{w}^{*}_1 + 
	\frac{1}{3}\Delta^3 + \Delta \hat{w}^{*2}_T - \Delta^2 \hat{w}^{*}_T + 
	\frac{1}{3}\Delta^3 + \\ 
	& ~~~~~~~~~~~~~~~ \sum_{i=1}^{T-1} \left( \Delta 
	(\hat{w}^{*}_i-\hat{w}^{*}_{i+1})^2 - \Delta^2 
	(\hat{w}^{*}_i-\hat{w}^{*}_{i+1}) + \frac{1}{3}\Delta^3 \right) -\gamma 
	\hat{w}_1 \\
	& = T(T+1)\Delta \delta^2-T(\gamma + 2\Delta^2)\delta 
	+\frac{(T+1)\Delta^3}{3} \\
	& = \frac{T(\gamma + 2\Delta^2)^2}{4\Delta(T+1)} - \frac{T\left(\gamma + 
	2\Delta^2\right)^2}{2\Delta(T+1)} +\frac{(T+1)\Delta^3}{3} \\
	& = - \frac{T\left(\gamma + 2\Delta^2\right)^2}{4\Delta(T+1)} 
	+\frac{(T+1)\Delta^3}{3}~,
	\end{align*}
	and	
	\[
	\norm{\hat{\bw}^*}^2 ~=~  \frac{\left(\gamma + 2\Delta^2\right) 
	^2}{4\Delta^2} \sum_{t=1}^T \left(1-\frac{t}{1+T} \right)^2~\leq~
	\frac{(T+1)\left(\gamma + 2\Delta^2\right) ^2}{12\Delta^2}
	\]
\end{proof}

\subsection{Oracle Complexity Lower Bound}\label{subsec:mainproof_convex}

Given the expressions on the optimal value of $\hat{f}_T$, derived in the 
previous subsection, we turn to explain how the oracle complexity lower bound 
is derived. The argument is very similar to the strongly convex case (proof of 
\thmref{thm:main_strong}, subsection \ref{subsec:mainproof}): Specifically, 
consider the function $f_{2T}$, given 
by
\[
f_{2T}(\bw) =
\frac{\mu_2}{12}\left(g(\inner{\bv_1,\bw})+g(\inner{\bv_{2T},\bw})+\sum_{i=1}^{2T-1}
g(\inner{\bv_i,\bw}-\inner{\bv_{i+1},\bw})-\gamma\inner{\bv_1,\bw}\right).
\]
Given 
an algorithm, we choose 
$\bv_1,\bv_2,\ldots,\bv_{T}$ to be orthogonal unit vectors, so that each 
$\bv_t$ is orthogonal 
to the first $t$ points
$\bw_{1},\bw_{2},\ldots,\bw_{t}$ computed by the algorithm (this is possible, 
since the gradients and Hessians of $f_{2T}$ at each $\bw_t$ reveals no 
information on future $\bv_t$'s -- see \lemref{lem:info}). Also, we let 
$\bv_{T+1},\ldots,\bv_{2T}$ equal $\bv_T$. 

With this choice, it is easily verified that 
\[
f_{2T}(\bw_T) = 
\frac{\mu_2}{12}\left(g(\inner{\bv_1,\bw})+g(\inner{\bv_{T},\bw})+\sum_{i=1}^{T-1}
g(\inner{\bv_i,\bw_T}-\inner{\bv_{i+1},\bw_T})-\gamma\inner{\bv_1,\bw_T}\right),
\]
which is clearly no better than $\min_{\bw} f_T(\bw)$, where $f_T$ is defined 
with the same $\bv_1,\ldots,\bv_T$. Therefore, we can lower bound the 
optimization error $f_{2T}(\bw_T)-\min_{\bw}f_{2T}(\bw)$ by 
$\min_{\bw}f_{T}(\bw)-\min_{\bw}f_{2T}(\bw)$. Moreover, by \eqref{eq:fhatf}, 
this equals
\[
\frac{\mu_2}{12}\left(\min_{\bw}\hat{f}_T(\bw)-\min_{\bw}\hat{f}_{2T}(\bw)\right).
\]
Using proposition \ref{prop:opt_convex}, we can now plug in these minimal 
values, depending on the various parameter regimes, and get an oracle 
complexity lower bound. Computing these bounds and parameter regimes (while 
picking the free parameter $\gamma$ appropriately) is performed in the next 
subsection. 

\subsection{Setting the $\gamma$ Parameter}\label{subsec:smoothconvex}

To simplify notation, we let $\hat{f}^*_{T}$ and $\hat{f}^*_{2T}$ be shorthand 
for $\min_{\bw}\hat{f}_T(\bw)$ and $\min_{\bw}\hat{f}_{2T}(\bw)$ respectively, 
with minimizers $\hat{\bw}^*_{T}$ and $\hat{\bw}^*_{2T}$. We will consider 
three regimes, depending on the relationships between $D,\Delta,T$.


\subsubsection{Case 1: $\frac{D^2}{48\Delta^2 T^3}\leq 
	\frac{1}{T^2}$}

In this setting, we choose
\[
\gamma = \frac{D^2}{48T}
\]
Using this and the assumption on the parameters, we get that $\gamma \leq \Delta^2 <
\frac{\Delta^2 \left( 1+4T^2\right)}{4T^2} < \frac{\Delta^2 \left( 1+T^2\right)}{T^2}$ , and therefore, we are in the 
first regime  for both $f_T$ and $f_{2T}$ as specified in proposition 
\ref{prop:opt_convex}. Plugging in 
the bound on $\norm{\hat{\bw^*}}^2$ in that regime, and using the fact that 
$\Delta^2\leq\gamma$ by the assumption above, we have
\[
\norm{\hat{\bw}_{2T}^*}^2_2 \leq \frac{\gamma(1+2T)^3}{3\left(1+4T^2\right)} 
\leq \frac{27D^2T^3}{576T^2} \leq D^2
\] as required.

Using the results from proposition \ref{prop:opt_convex} for the first regime 
we can compute the optimization error bound
\begin{align*}
\hat{f}_{T}^* - \hat{f}_{2T}^* &= \frac{4
	\gamma^{3/2}T}{3\sqrt{(1+4T^2)}}-\frac{2 \gamma^{3/2}T}{3\sqrt{(1+T^2)}} \\
& = \frac{ 2\gamma^{3/2}}{3} \left( \frac{1}{\sqrt{\left(1 + 
\frac{1}{4T^2}\right)}}- \frac{1}{\sqrt{\left(1+\frac{1}{T^2}\right)}} \right) 
\\
& \geq \frac{2\gamma^{3/2}}{3} \left( 1 - \frac{1}{8T^2} - \left( 1 - 
\frac{1}{2T^2} + \frac{3}{8T^4} \right) \right) \\
& = \frac{ 2\gamma^{3/2}}{3}  \left(  \frac{3}{8T^2} - \frac{3}{8T^4} \right) = 
\frac{ \gamma^{3/2}\left( T^2-1 \right)}{4T^4} \\
& \geq \frac{ D^3}{1331T^{7/2}} \\
\end{align*}

Where in the first inequality we used the fact that
$
1 - \frac{1}{2}x \leq \frac{1}{\sqrt{1+x}} \leq 1 - \frac{1}{2}x + 
\frac{3}{8}x^2
$ for all $x\geq 0$ and for the last inequality we assumed that $T \geq 2$. In the case that $T=1$, the final result still holds.  
Hence, the suboptimality is at least $\frac{\mu_2D^3}{16000 T^{7/2}}$.

\subsubsection{Case 2: $\frac{1}{T^2} < 
	\frac{D^2}{48\Delta^2 T^3} \leq 1$}

In this setting, we choose 
\begin{equation}\label{eq:gammaval}
\gamma = {\frac{D\Delta}{ \sqrt{12T}}}.
\end{equation}
Using this and the assumption on the parameters, we get that $\frac{\Delta^2 
	\left( 1+T^2\right)}{T^2} <\gamma < 2 \Delta^2T$, and therefore, we are in 
	the 
second regime for both $f_T$ and $f_{2T}$ as specified in proposition 
\ref{prop:opt_convex}. Plugging in 
the bound on $\norm{\hat{\bw^*}}^2$ in that regime, and using the fact that 
$\Delta^2<\gamma$ by the assumption above, we have
\[
\norm{\hat{\bw}^*_{2T}}^2 ~\leq~ \frac{\left(\gamma + \Delta^2\right)^2(2T+1)^3 
}{48\Delta^2T^2} ~\leq~ \frac{\gamma^2(2T+1)^3
}{12\Delta^2T^2} ~=~ \frac{D^2(2T+1)^3
}{144T^3} ~\leq~D^2
\] as required.

Turning to compute the optimization error bound, and letting 
$\delta_T,\delta_{2T}$ denote the quantity $\delta$ in proposition 
\ref{prop:opt_convex} for $\hat{f}_T$ and $\hat{f}_{2T}$ respectively, we have
\begin{equation}\label{eq:subopt1}
f^*_T-f^*_{2T} ~=~ \left(2\delta_{2T}-\delta_T \right)\left(T\left( 
\Delta^2+\gamma \right)-\Delta T^2 \left(2\delta_{2T}+\delta_T \right) 
\right)+\frac{1}{3}T \left(\delta_T^3-2\delta_{2T}^3 \right) .
\end{equation}
To continue, we use the following auxiliary lemma:
\begin{lemma}
	$\left(2\delta_{2T}-\delta_T \right)\left(T\left( \Delta^2+\gamma 
	\right)-\Delta T^2 \left(2\delta_{2T}+\delta_T \right) \right)\geq 0$
\end{lemma}
\begin{proof}

First we will prove that $T\left( \Delta^2+\gamma 
\right)-\Delta T^2 \left(2\delta_{2T}+\delta_T \right) \geq 0$.

Since 	$\delta_T = -\Delta T + \Delta T\sqrt{1+\frac{\gamma + \Delta^2 
}{\Delta^2 T^2}}$
and using $\sqrt{1+x} \leq 1+ \frac{1}{2}x$ for $x\geq 0$
we have that:
	\[\delta_T \leq \frac{\left(\gamma + 
	\Delta^2\right)}{2\Delta T} \]
So
\[
T\left( \Delta^2+\gamma \right)-\Delta T^2 \left(2\delta_{2T}+\delta_T \right) 
\geq T\left( \Delta^2+\gamma \right) -\Delta T^2 \frac{\left(\gamma + 
\Delta^2\right)}{\Delta T} = 0
\]
To complete the proof, it remains to show that $2\delta_{2T} -\delta_T \geq 0$. 
We have
\[
2\delta_{2T} -\delta_T ~=~ -4\Delta T + 4\Delta T\sqrt{1+\frac{\gamma + 
\Delta^2 }{4\Delta^2 T^2}}  +\Delta T - \Delta T\sqrt{1+\frac{\gamma + \Delta^2 
}{\Delta^2 T^2}}
\]
Define $\alpha := \frac{\gamma + \Delta^2 }{\Delta^2 T^2} \geq 0$.
Hence, we need to prove:
\begin{align*}
&-4+4\sqrt{1+\frac{1}{4}\alpha}+1-\sqrt{1+\alpha} \geq 0\\
&\iff 4\sqrt{1+\frac{1}{4}\alpha} \geq 3+\sqrt{1+\alpha}\\
&\iff 16\left(1+\frac{1}{4}\alpha\right) \geq 9+6\sqrt{1+\alpha}+1+\alpha\\
&\iff 6 + 3\alpha \geq 6\sqrt{1+\alpha}\\
\end{align*}
Which is true since $\sqrt{1+\alpha} \leq 1+\frac{1}{2}\alpha$.
\end{proof}
With this lemma, we can lower bound the optimization error in 
\eqref{eq:subopt1} by 
\begin{equation}\label{eq:subopt2}
\frac{1}{3}T \left(\delta_T^3-2\delta_{2T}^3 \right). 
\end{equation}
To continue, we note that by definition of 
$\delta_T,\delta_{2T}$ and the fact that $1+ \frac{1}{2}x -\frac{1}{8}x^2  \leq 
\sqrt{1+x} \leq 1+ \frac{1}{2}x$, we have
\[
\frac{\left(\gamma + \Delta^2\right)}{2\Delta T} - \frac{\left(\gamma + 
	\Delta^2\right)^2}{8\Delta^3 T^3} \leq \delta_T \leq \frac{\left(\gamma + 
	\Delta^2\right)}{2\Delta T}.
\]
Therefore, 
\begin{align*}
\delta_T-\sqrt[3]{2}\delta_{2T} &~\geq~\frac{\left(\gamma + 
	\Delta^2\right)}{2\Delta T} - \frac{\left(\gamma + 
	\Delta^2\right)^2}{8\Delta^3 
	T^3} - \frac{\sqrt[3]{2}\left(\gamma+\Delta^2\right)}{4\Delta T} \\
& ~\geq~ \frac{\left(\gamma + \Delta^2\right)}{20\Delta T} + \frac{\left(\gamma 
	+ \Delta^2\right)}{8\Delta T} \left(1 - \frac{\gamma + \Delta^2}{\Delta^2 
	T^2} 
\right)\\
& ~\geq~ \frac{\left(\gamma + \Delta^2\right)}{20\Delta T}.
\end{align*}
Using this inequality, and the fact $(a-b)^3 \leq a^3-b^3$ 
for $a\geq b \geq 0$, we can lower bound \eqref{eq:subopt2} by 
\[
\frac{1}{3}T \left(\delta_T-\sqrt[3]{2}\delta_{2T}\right)^3  \geq  
\frac{\left(\gamma+\Delta^2\right)^3}{60\Delta^3 T^2} 
\geq \frac{D^3}{2500 T^{7/2}}.
\]
Hence, the suboptimality is at least $\frac{\mu_2D^3}{30000 T^{7/2}}$.

\subsubsection{Case 3: $\frac{D^2}{48\Delta^2 T^3} > 1$ }
In this setting, we choose
\[
\gamma = {\frac{D\Delta}{ \sqrt{3T}}}.
\]
Using this and the assumption on the parameters, we get that $\gamma > 4 
\Delta^2T$, and therefore, we are in the 
third regime for both $f_T$ and $f_{2T}$ as specified in proposition 
\ref{prop:opt_convex}. Plugging in 
the bound on $\norm{\hat{\bw}^*_{2T}}^2$ in that regime, and using the fact 
that 
$2\Delta^2<\gamma$ by the assumption above, we have
\[
\norm{\hat{\bw}^*_{2T}}^2 ~\leq~ \frac{\left(\gamma + 2\Delta^2\right) 
^2(2T+1)}{12\Delta^2} ~\leq~ \frac{4\gamma^2(2T+1)
}{12\Delta^2} = \frac{D^2(2T+1)}{9T} \leq D^2
\]
Now, by the assumptions that  $T\Delta^3<\frac{\Delta D^2}{48T^2}$ and by using  the fact that
$
1 - x \leq \frac{1}{1+x} \leq 1 - x + x^2$ for all $x\geq 0$, the optimization error bound is  
\begin{align*}
\hat{f}_T^*-\hat{f}_{2T}^* & \geq  \frac{2T\left(\gamma + 2\Delta^2\right)^2}{4\Delta(2T+1)} - \frac{T\left(\gamma + 2\Delta^2\right)^2}{4\Delta(T+1)} 
+\frac{(T+1)\Delta^3}{3}  
-\frac{(2T+1)\Delta^3}{3} \\
& =\frac{\left(\gamma + 2\Delta^2\right)^2}{4\Delta} \left(
\frac{1}{1+\frac{1}{2T}} -\frac{1}{1+\frac{1}{T}}\right) -\frac{T\Delta^3}{3}\\
& \geq \frac{D^2\Delta}{12T} \left(\frac{1}{2T} - \frac{1}{T^2}\right) -\frac{T\Delta^3}{3} \geq \frac{D^2\Delta}{72T^2} - \frac{ D^2\Delta}{144T^2}\\
& =\frac{ D^2\Delta}{144T^2}\\
\end{align*}
In the last inequality we assumed that $T \geq 3$. For $T=1,2$ it can be easily verified that the inequality holds.
Hence, using $\Delta = \frac{3\mu_1}{2\mu_2}$ the suboptimality is at least $\frac{\mu_1D^2}{576T^{2}}$.

\subsection{Wrapping Up}
Combining the three cases from the previous subsection, we see that we get  the 
following lower bound
\[
f_{2T}(\bw_T)-\min_{\bw}f_{2T}(\bw)  \geq \begin{cases} \frac{\mu_2D^3}{30000 
T^{7/2}} & \frac{D^2}{48\Delta^2 T^3} \leq  1\\
\frac{\mu_1D^2}{576T^{2}} &\frac{D^2}{48\Delta^2 T^3} >  1 
\end{cases}~.
\]
Thus, we get that
\[
f_{2T}(\bw_T)-\min_{\bw}f_{2T}(\bw) \geq \min\left\{\frac{\mu_2D^3}{30000 
T^{7/2}}, \frac{\mu_1D^2}{576T^{2}} \right\}~.
\]
Equating these bounds to $\epsilon$, and solving 
for $T$, the theorem follows.

\section{Proof of \thmref{thm:main3}}
The proof of \thmref{thm:main3} will follow the same outline of the proof of \thmref{thm:main2}. 
We are again going to assume without loss of generality that $\bw_1=0$, and we will thus require that $\norm{\bw^*}\leq D$
(see discussion in the proof of \thmref{thm:main2}).
We  define $g:\reals\mapsto\reals$ as
\[
g(x) =  \frac{1}{k+1}|x|^{k+1} 
\]
and
\[
f_T(\bw) =
\frac{\mu_{k}}{k!2^{\frac{k+3}{2}}}\left(g(\inner{\bv_1,\bw})+g(\inner{\bv_T,\bw})+\sum_{i=1}^{T-1}g(\inner{\bv_i,\bw}-\inner{\bv_{i+1},\bw})-\gamma\inner{\bv_1,\bw}\right).
\]
By the following lemma, $f_T(\bw)$ is $k$-times differentiable, with
$\mu_k$-Lipschitz $k-th$  order derivative tensor. 

\begin{lemma}\label{lem:smoothHigherOrder}
	$f_T(\bw)$ is $k$-times differentiable, with
	$\mu_k$-Lipschitz $k-th$  order derivative tensor. 
\end{lemma}
\begin{proof}
	Similarly to \lemref{lem:smoothstrong}, we can assume without loss of generality, that the vectors $\bv_1,\bv_2,...,\bv_T$ correspond to the standard basis vectors $\be_1,\be_2,...,\be_T$, so we can examine the Lipschitz property of 
	\[
	\hat{f}(\bw) =
	\frac{\mu_{k}}{k!2^{\frac{k+3}{2}}}\left(g(w_1)+g(w_T)+\sum_{i=1}^{T-1}g(w_i-w_{i+1})-\gamma w_1\right).
	\]
	Let 
	\[
	\br_i := \begin{cases} \be_1 & i = 0\\ 
	\be_i - \be_{i+1} & 1 \leq i  \leq T-1  \\
	\be_T & i = T
	\end{cases}.
	\]
	
	By Differentiating $k$ times, we have that 
	\begin{align*}
	\nabla^{(k)}\hat{f}(\bw) &=  \frac{\mu_{k}}{k!2^{\frac{k+3}{2}}}\left(g^{(k)}(w_1)\br_0^{\otimes k}  +  g^{(k)}(w_T)\br_T^{\otimes k}  + 
	\sum_{i=1}^{T-1}  g^{(k)}(\inner{\br_i,\bw})\br_i^{\otimes k}\right) \\ 
	& = \frac{\mu_{k}}{k!2^{\frac{k+3}{2}}}\left( 
	\sum_{i=0}^{T}  g^{(k)}(\inner{\br_i,\bw})\br_i^{\otimes k}\right), \\ 
	\end{align*}
	Where
	\[
	\bv^{\otimes p} = \underbrace{\bv \otimes \bv \otimes \cdots \bv}_{\text{$p$ times}}
	\]
	
	Since $g^{(k)}(x) = k!x$
	\begin{align}\label{eq:norm_bound_higher_order}
	\left\|\nabla^{(k)}\hat{f}(\bw) - \nabla^{(k)}\hat{f}(\tilde{\bw})\right\| & = \frac{\mu_{k}}{k!2^{\frac{k+3}{2}}}  \left\|\sum_{i=0}^{T}  \left(g^{(k)}(\inner{\br_i,\bw})-g^{(k)}(\inner{\br_i,\tilde{\bw}})\right)\br_i^{\otimes k}\right\|  \\
	& =  \frac{\mu_{k}}{2^{\frac{k+3}{2}}} \left\|\sum_{i=0}^{T}  \inner{\br_i,\bw-\tilde{\bw}}\br_i^{\otimes k}\right\| 
	\end{align}
	
	Note that for a $k$-th order symmetric tensor $T$ , the operator norm equals (see e.g. \citep{ mu2015successive}):
	\[
	\norm{T} = \max_{\norm{\bx}=1} \left|\sum_{i_1,i_2,...,i_k}T_{i_1,i_2,...,i_k}x_{i_1}x_{i_2}...x_{i_k}\right|
	\]
	
	So,
	\begin{align*}
	\left\|\sum_{i=0}^{T}  \inner{\br_i,\bw-\tilde{\bw}}\br_i^{\otimes k}\right\| & = \max_{\norm{\bx}=1} \left| \sum_{i_1,i_2,...,i_k}\sum_{i=0}^{T}\inner{\br_i,\bw-\tilde{\bw}}r_{i,i_1}r_{i,i_2}...r_{i,i_k}x_{i_1}x_{i_2}...x_{i_k}\right|\\
	& = \max_{\norm{\bx}=1}\left| \sum_{i=0}^{T}\inner{\br_i,\bw-\tilde{\bw}}\sum_{i_1}r_{i,i_1}x_{i_1}...\sum_{i_k}r_{i,i_k}x_{i_k} \right| \\
	& =  \max_{\norm{\bx}=1}\left| \sum_{i=0}^{T}\inner{\br_i,\bw-\tilde{\bw}}\inner{\br_i,\bx}^k \right| \leq 2^{\frac{k}{2}}\max_{\norm{\bx}=1}\sum_{i=0}^{T}\left| \inner{\br_i,\bw-\tilde{\bw}} \right| \left| \left\langle \frac{\br_i}{\sqrt{2}},\bx \right\rangle \right|^k\\
	& \leq  2^{\frac{k+1}{2}}\norm{\bw-\tilde{\bw}} \max_{\norm{\bx}=1}\sum_{i=0}^{T} \left\langle \frac{\br_i}{\sqrt{2}},\bx \right\rangle ^2\\
	& = 2^{\frac{k-1}{2}} \norm{\bw-\tilde{\bw}} \max_{\norm{\bx}=1} x_1^2 + x_T^2 +\sum_{i=1}^{T-1} \left( x_i - x_{i+1}\right)^2 \\
	& \leq 2^{\frac{k-1}{2}} \norm{\bw-\tilde{\bw}} \max_{\norm{\bx}=1} x_1^2 + x_T^2 +2\sum_{i=1}^{T-1}  x_i^2 + 2\sum_{i=2}^{T} x_i^2 \leq  2^{\frac{k+3}{2}} \norm{\bw-\tilde{\bw}}\\
	\end{align*}
	Where in the first inequality we used that $\norm{\br_i} \leq \sqrt{2}~~~\forall i$. 
	
	Plugging that into \eqref{eq:norm_bound_higher_order} we have that
	\[
	\left\|\nabla^{(k)}\hat{f}(\bw) - \nabla^{(k)}\hat{f}(\tilde{\bw}\right\| \leq \mu_{k} \norm{\bw-\tilde{\bw}}
	\]
	as required.
\end{proof}

\subsection{Minimizer of $f_T$}
In order to derive the complexity bound, we will first analyze $\hat{f}_T$, which is a simplified version of $f_T$, as defined in \subsecref{subsec:min_f_T_convex}.
It is easily verified that $\min_\bw f_T(\bw)=\frac{\mu_{k}}{k!2^{\frac{k+3}{2}}}\cdot \min_{\bw}\bw\hat{f}_T(\hat{\bw})$, and moreover, if $\hat{\bw}\in\reals^T$ is a minimizer of $\hat{f}_T$, 
then $\bw^*=\sum_{j=1}^{T}\hat{w}^*_j\cdot \bv_j\in \reals^d$ is a minimizer of 
$f_T$, and with the same Euclidean norm as $\hat{\bw}^*$.

Using an identical proof to \lemref{lem:gen_sol} we can have that
$\hat{f}_T$ has a unique minimizer $\hat{\bw}^* \in \reals^T$, which  
satisfies
\[
\hat{w}^*_t = \delta\cdot(T+1)\cdot\left(1-\frac{t}{T+1}\right)~,
\]
for some $\delta > 0 $ and all $t=1,2,\ldots,T$
and 
\[
g'(w^*_1) + g'(w^*_T) = (\delta T)^{k} + \delta^{k}= \delta^{k}(T^{k} + 1) = \gamma 
\]
hence,
\[
\delta = \left( \frac{\gamma}{T^{k} + 1} \right)^\frac{1}{k}.
\]
By plugging the minimizer, we have that

\begin{align}\label{eq:min_value_higher_order}
\hat{f}_T(\hat{\bw}^*) & = g(w_1^*) + g(w_T^*) + \sum_{i=1}^{T-1}g(w_i^*-w_{i+1}^*)-\gamma w_1^* \nonumber\\
& = \frac{1}{k+1}(\delta T)^{k+1} + \frac{T}{k+1}\delta^{k+1} -\gamma \delta T \nonumber\\
& = \frac{T^{k+1}}{k+1}\left( \frac{\gamma}{T^{k} + 1} \right)^\frac{k+1}{k} + \frac{T}{k+1}\left( \frac{\gamma}{T^{k} + 1} \right)^\frac{k+1}{k} - \gamma T \left( \frac{\gamma}{T^{k} + 1} \right)^\frac{1}{k} \nonumber\\
& = \frac{T}{k+1}\left(T^{k}+1 \right)\left( \frac{\gamma}{T^{k} + 1} \right)^\frac{k+1}{k}  - \gamma T \left( \frac{\gamma}{T^{p-1} + 1} \right)^\frac{1}{p-1} \nonumber\\
& = \frac{T}{k+1} \frac{\gamma^\frac{k+1}{k}}{\left(T^{k} + 1\right)^\frac{1}{k}}  -  T  \frac{\gamma^\frac{k+1}{k}}{\left(T^{k} + 1\right)^\frac{1}{k}} \nonumber\\
& = - \frac{kT\gamma^\frac{k+1}{k}}{(k+1)\left(T^{k} + 1\right)^\frac{1}{k}}
\end{align}
and,

\begin{align}\label{eqn:norm_bound_higherorder}
\norm{\hat{\bw}^*}^2_2 = & \sum_{t=1}^T\hat{w}^{*2}_t 
=\delta^2 \left(1+T\right)^2\sum_{t=1}^T \left(1-\frac{t}{1+T} 
\right)^2 \nonumber \\
& \leq \frac{1}{3}\left( \frac{\gamma}{T^{k} + 1} \right)^\frac{2}{k} \left(1+T\right)^3
\end{align}

Where we used $\sum_{t=1}^T \left(1-\frac{t}{1+T}\right)^2  \leq \frac{1}{3}(1+T)$ as in \propref{prop:opt_convex}.
\subsection{Oracle Complexity Lower Bound}\label{subsec:mainproof_convex_higherorder}
The derivation of the lower complexity bound will be exactly the same as in \subsecref{subsec:mainproof_convex}.

In \subsecref{subsec:mainproof_convex} we showed that we can lower bound the 
optimization error $f_{2T}(\bw_T)-\min_{\bw}f_{2T}(\bw)$ by 
$\min_{\bw}f_{T}(\bw)-\min_{\bw}f_{2T}(\bw)$. Using the fact that
\[
f_T(\bw^*)=\frac{\mu_{k}}{k!2^{\frac{k+3}{2}}}\cdot\hat{f}_T(\hat{\bw}^*)~,
\] 
this equals
\[
\frac{\mu_{k}}{k!2^{\frac{k+3}{2}}}\left(\min_{\bw}\hat{f}_T(\bw)-\min_{\bw}\hat{f}_{2T}(\bw)\right).
\]

Letting $f_T^*$ and $\hat{f}_T^*$ to be the minimal values of $f_T$ and $\hat{f}_T$ respectively, and by using equation \eqref{eq:min_value_higher_order} then

\begin{align*}
\hat{f}_{T}^* - \hat{f}_{2T}^* &= \frac{k\gamma^\frac{k+1}{k}}{(k+1)\left(1 + \frac{1}{(2T)^{k}}\right)^\frac{1}{k}}- \frac{k\gamma^\frac{k+1}{k}}{(k+1)\left(1 + \frac{1}{T^{k}}\right)^\frac{1}{k}} \\
& \geq \frac{k\gamma^\frac{k+1}{k}}{k+1}\left( 1 - \frac{1}{k(2T)^{k}} - \left( 1 - \frac{1}{kT^{k}} + \frac{k+1}{2k^2 T^{2k}} \right)\right) \\ 
& = \frac{\gamma^\frac{k+1}{k}}{(k+1)kT^{k}}\left( 1 - \frac{1}{2^{k}} - \frac{k+1}{2kT^{k}}\right) \\ 
& \geq  \frac{\gamma^\frac{k+1}{k}}{6(k+1)kT^{k}} \\
\end{align*}

The last inequality holds for $k=1, T \geq 3$, $k = 2, T\geq 2$ or $ k  \geq 3, T \geq 1$.  It can be verified that for the other cases, the inequality above holds.

Since we want $f_{T}^* - f_{2T}^*$ to be as large as possible, we will set $\gamma$ to be as large as possible, under the constraint that $\norm{\bw^*_{2T}} \leq D$. By \eqref{eqn:norm_bound_higherorder} we can choose
\[
\gamma = \frac{3^{\frac{k}{2}}D^{k}\left(1+ (2T)^{k} \right)}{(1+2T)^{\frac{3k}{2}}}
\]

\begin{align*}
\hat{f}_{T}^* - \hat{f}_{2T}^* & \geq \frac{3^{\frac{k+1}{2}}D^{k+1}\left(1+ (2T)^{k} \right)^{\frac{k+1}{k}}}{6(k+1)kT^{k}(1+2T)^{\frac{3(k+1)}{2}}} \\
& \geq \frac{2^{k+1} D^{k+1}}{6\cdot3^{k+1}(k+1)kT^{\frac{3k+1}{2}}}
\end{align*}
Thus, according to the discussion in \subsecref{subsec:mainproof_convex_higherorder}, the final bound is 
\[
f_{T}^* - f_{2T}^* \geq \frac{\mu_{k}\sqrt{2}^{k+1} D^{k+1}}{12\cdot3^{k+1}(k+1)!kT^{\frac{3k+1}{2}}}
\]
and the number of iterations required for having $\min_{\bw}f_{T}(\bw)-\min_{\bw}f_{2T}(\bw) < \epsilon$ , $T_\epsilon$
must satisfy 
\begin{align*}
T_\epsilon & \geq \left(\frac{\mu_{k}\sqrt{2}^{k+1} D^{k+1}}{12\cdot3^{k+1}(k+1)!k\epsilon}\right)^\frac{2}{3k+1} \\
& \geq c\left(\frac{\mu_{k} D^{k+1}}{(k+1)!k\epsilon}\right)^\frac{2}{3k+1}
\end{align*}
Where $c = \left(\frac{1}{12}\right)^{\frac{1}{5}}\cdot 
\left(\frac{\sqrt{2}}{3}\right)^\frac{4}{5}$.

\subsection*{Acknowledgments}

We thank Yurii Nesterov for several helpful comments on a preliminary version 
of this paper, as well as Naman Agarwal, Elad Hazan and Zeyuan Allen-Zhu for 
informing us about the A-NPE algorithm of \citet{monteiro2013accelerated}.

\bibliographystyle{plainnat}
\bibliography{bib}

\appendix

\section{An Improved Second-Order Oracle Complexity Bound for Strongly 
Convex Functions}
\label{sec:opt_strongly_convex_alg}

In this section, we show how the A-NPE algorithm of  
\cite{monteiro2013accelerated}, which is a second-order method analyzed for 
smooth convex functions, can 
be used to yield near-optimal performance if the function is also strongly 
convex. Rather than directly adapting their analysis, which is non-trivial, we 
use a 
simple restarting scheme, which allows one to convert an 
algorithm for the convex setting, to an algorithm in the strongly convex 
setting\footnote{We note that the \emph{reverse} direction, of adapting 
strongly convex 
optimization algorithms to the convex case, is more common in the 
literature, and can be achieved using regularization or more sophisticated 
approaches \citep{allen2016optimal}.}.

Our algorithm is described as follows: In the first phase, we apply a generic 
restarting scheme (based on \cite[Subsction 4.2]{arjevani2016iteration}), where 
we repeatedly run A-NPE for a bounded number of steps, followed by restarting 
the algorithm, running it from the last iterate obtained. By strong convexity, 
we show that each such epoch reduces the suboptimality by a constant factor. 
Once we reach a point sufficiently close to the global optimum, we switch to 
the second phase, where we use the cubic-regularized Newton method to get a 
quadratic convergence rate.

To formalize this, let us first analyze the convergence rate of the first phase.
We assume that we use the algorithm described in \citet[Subsection 
7.4]{monteiro2013accelerated}\footnote{Specifically, since in 
our framework we do not limit computational resources, we assume that the 
minimization problem in Eq. (6.1) of \citet{monteiro2013accelerated} can be 
solved 
exactly.}. By \cite[Theorem 6.4 and Theorem 3.10]{monteiro2013accelerated}, we 
have that the $t$'th iterate satisfies
\begin{align*}
	\norm{\bw_t-\bw_*}\le D ~~\text{and }~~ f(\bw_t) - f(\bw^*) \le \frac{c 
	\mu_2 
	\norm{\bw_1-\bw^*}^3}{t^{7/2}},
\end{align*}
where $\mu_2$ is the Lipschitz constant of $\nabla^2 f$, $\bw_1$ is the
initialization point, $\bw^*$ is the unique minimizer (due to strong convexity) 
of 
$f$, $D$ bounds $\norm{\bw_1-\bw^*}$ from above, and $c>0$ is some universal 
constant. Since $f$ is also assumed to be $\lambda$-strongly convex, we have
\begin{align*}
	\frac{\lambda}{2}\norm{\bw_1 - \bw^*}^2 &\le f(\bw_1) - f(\bw^*),
\end{align*}
hence
\begin{align*}
	f(\bw_t) - f(\bw^*) \le \frac{c \mu_2 \norm{\bw_1-\bw^*}^3}{t^{7/2}} \le  
	\frac{2c \mu_2 \norm{\bw_1-\bw^*}(f(\bw_1)-f(\bw^*))}{\lambda t^{7/2}} 
	\le \frac{2c \mu_2 D(f(\bw_1)-f(\bw^*))}{\lambda t^{7/2}}~.
\end{align*}
Thus, running the algorithm for 
\begin{align*}
	\tau &= \left(\frac{4c\mu_2 D}{\lambda}\right)^{2/7}
\end{align*}
iterations, we see that $f(\bw_t) - f(\bw^*) \le {(f(\bw_1)-f(\bw^*))}/{2}$. 
Now, since 
the distance from $\bw_t$ to $\bw^*$ is also smaller than $D$, we may 
initialize the algorithm at the last iterate returned by the previous run and 
run it for $\tau$ iterations to reduce $f(\bw_t) - f(\bw^*)$ in, yet again, a 
factor 
of 2. Applying the algorithm for $T$ iterations (and restarting the algorithmic 
parameters after every $\tau$ iterations) yields
$$f(\bw_T)-f(\bw^*) \le \frac{f(\bw_1)-f(\bw^*) }{2^{T/\tau}}.$$
Equivalently, to obtain an $\epsilon$-optimal solution, we need at most
\begin{align*}
	\left(\frac{4c\mu_2 
	D}{\lambda}\right)^{2/7}\log_2\left(\frac{f(\bw_1)-f(\bw^*) 
	}{\epsilon}\right)
\end{align*}
oracle calls (note that this restarting scheme can be applied also on uniform 
convex functions of any order (defined in, e.g., 
(\cite{vladimirov1978uniformly})).\\

Next, after performing a number of iterations sufficiently large to obtain high 
accuracy solutions, we proceed to the second phase of the algorithm where 
cubic-regularized Newton steps are applied (see 
\citet{nesterov2008accelerating}). According to that analysis, after reducing 
the optimization error to below $\lambda^3/4\mu_2^2$, the number of 
cubic-regularized Newton steps required to 
achieve an $\epsilon$-suboptimal solution is
\begin{align*}
	\Ocal\left(\log\log_2\left(\frac{\lambda^3}{\mu_2^2\epsilon}\right)\right).
\end{align*}
Thus,  using the $\mu_1$-Lipschitzness of the gradient to bound 
$f(\bw_1)-f(\bw^*)$ 
from above by $\mu_1 D^2/2$, we get that the overall number of iterations is at 
most
\begin{align*}
	\Ocal\left(\left(\frac{\mu_2 
	D}{\lambda}\right)^{2/7}\log_2\left(\frac{\mu_1 
	\mu_2^2D^2 }{\lambda^3}\right) + 
	\log\log_2\left(\frac{\lambda^3}{\mu_2^2\epsilon}\right)\right)~.
\end{align*}

\end{document}